\documentclass{article}
\usepackage{times}
\usepackage{graphicx}
\usepackage{amsmath,amsthm}
\usepackage{amssymb}
\usepackage{dsfont}
\usepackage[usenames,dvipsnames]{xcolor}
\usepackage{enumitem}
\usepackage[boxed,linesnumbered]{algorithm2e}
\usepackage{thmtools}
\usepackage{tikz}
\usepackage{hyperref}

\usetikzlibrary{arrows,calc,positioning,intersections}

\usepackage[left=2cm,right=2cm,top=3cm,bottom=3cm]{geometry}

\newcommand{\R}{\mathbb{R}}
\newcommand{\N}{\mathbb{N}}

\newcommand{\Id}{{\mathds{1}}} 
\newcommand{\tr}{\mathrm{tr}} 
\DeclareMathOperator*{\argmin}{argmin}

\renewcommand{\d}{{\,\mathrm{d}}} 

\renewcommand{\div}{{\,\mathrm{div}}} 
\DeclareMathOperator{\cof}{cof}

\newcommand{\energy}{\mathcal{W}} 
\newcommand{\Energy}{\mathbf{W}}

\newcommand{\pathenergy}{\mathcal{E}}
\newcommand{\Pathenergy}{\mathbf{E}}

\newcommand{\metric}{g}

\newcommand{\manifold}{\mathcal{M}}

\newcommand{\y}{y}

\newcommand{\Velocity}{\mathbf{v}} 

\newcommand{\domain}{\Omega}

\newcommand{\Mass}{\mathbf{M}_h}
\newcommand{\Stiff}{\mathbf{S}_h}

\newcommand{\Exp}[1]{\mathrm{EXP}^{#1}} 

\newcommand{\image}{u}

\newcommand{\Image}{\mathbf{U}}

\newcommand{\admset}{\mathcal{A}} 
\newcommand{\deformation}{\phi}

\newcommand{\notinclude}[1]{}

\newcommand{\beq}{\begin{equation*}}
\newcommand{\eeq}{\end{equation*}}
\newcommand{\beqn}{\begin{equation}}
\newcommand{\eeqn}{\end{equation}}
\newcommand{\beqa}{\begin{eqnarray*}}
\newcommand{\eeqa}{\end{eqnarray*}}
\newcommand{\beqan}{\begin{align}}
\newcommand{\eeqan}{\end{align}}

\makeatletter
\DeclareRobustCommand\onedot{\futurelet\@let@token\@onedot}
\def\@onedot{\ifx\@let@token.\else.\null\fi\xspace}

\def\ie{\emph{i.e}\onedot} 
\def\cf{\emph{cf}\onedot} 
 
\def\wrt{w.r.t\onedot} 
\def\etal{\emph{et al}\onedot}

\makeatother

\makeatletter
\def\namedlabel#1#2{\begingroup#2\def\@currentlabel{#2}\phantomsection\label{#1}\endgroup}
\makeatother

\definecolor{uniblau}{HTML}{004291}
\definecolor{bigsblau}{HTML}{365079}
\definecolor{bigsblau50}{HTML}{9CA7BC}
\definecolor{bigsblau25}{HTML}{CDD3DD}
\definecolor{uniorangedark}{HTML}{E6B400}
\definecolor{uniorange}{HTML}{FFCB0E}
\definecolor{uniorange!50}{HTML}{FFE586}
\definecolor{uniwhite}{HTML}{FFF2C2}
\definecolor{hcmgruen}{HTML}{567877}
\definecolor{hcmgruen50}{HTML}{AFBDBE}
\definecolor{hcmgruen25}{HTML}{D7DEDE}
\definecolor{himgrau}{HTML}{626566}
\definecolor{himgrau75}{HTML}{949592}
\definecolor{himgrau50}{HTML}{C5C4BE}
\definecolor{himgrau25}{HTML}{F5F5F5}
\definecolor{textgrau}{HTML}{000000}
\definecolor{black}{HTML}{000000}
\definecolor{white}{HTML}{FFFFFF}
\definecolor{lightgray}{gray}{0.9}
\colorlet{hcmrot}{BrickRed}
\colorlet{red}{BrickRed}
\colorlet{blue}{uniblau}
\colorlet{greyblue}{bigsblau}
\colorlet{hcmgelb}{uniorange}
\colorlet{yellow}{uniorange}
\colorlet{tafelgruen}{hcmgruen}
\colorlet{green}{hcmgruen}

\newcommand{\testImage}{v}
\newcommand{\testDeformation}{\psi}
\newcommand{\testDeformationTwo}{\zeta}

\newcommand{\operator}{\mathcal{T}}
\newcommand{\Operator}{{\mathbf{T}}}
\newcommand{\secondOperator}{\mathcal{R}}
\newcommand{\SecondOperator}{\mathbf{R}}
\newcommand{\TestDeformation}{\mathbf{\Psi}}

\newcommand{\Deformation}{\mathbf{\Phi}}

\newcommand{\IntensityModulation}{\mathbf{I}}

\newcommand{\sym}{\mathrm{sym}}

\newcommand{\testSpace}{H^{2m}_0(\domain)}
\newcommand{\testSpaceDual}{H^{-2m}(\domain)}

\newcommand{\deformationNeighborhood}{{\bf{D}}}
\newcommand{\imageNeighborhood}{{\bf{U}}}

\newcommand{\IFTFunctional}{\mathcal{J}}
\newcommand{\detTerm}{{h}}

\newcommand{\quadratureWeight}{\omega}
\newcommand{\quadraturePoint}{\mathbf{x}}

\newcommand{\DeformationSpace}{{\mathcal{S}_H}}
\newcommand{\ImageSpace}{{\mathcal{V}_h}}
\newcommand{\indexImageNode}{{I_\ImageSpace^N}}

\numberwithin{equation}{section}

\theoremstyle{plain}
\newtheorem{theorem}{Theorem}[section]
\newtheorem{lemma}[theorem]{Lemma}
\newtheorem{proposition}[theorem]{Proposition}

\theoremstyle{remark}
\newtheorem*{remark}{Remark}

\theoremstyle{definition}

\begin{document}

\title{Image Extrapolation for the Time Discrete \\ Metamorphosis Model -- Existence and Applications}
\author{Alexander Effland, Martin Rumpf and Florian Sch\"afer}
\maketitle

\begin{abstract}
The space of images can be equipped with a Riemannian metric
measuring both the cost of transport of image intensities and
the variation of image intensities along motion lines. The resulting metamorphosis model was introduced 
and analyzed in \cite{MiYo01,TrYo05} and a variational time discretization for the geodesic interpolation was proposed in \cite{BeEf14}.
In this paper, this time discrete model is expanded and an image extrapolation via a discrete exponential map is consistently derived for the variational time discretization.
For a given weakly differentiable initial image and an initial image variation, the exponential map allows to compute a discrete geodesic
extrapolation path in the space of images. It is shown that a time step of this shooting method can be formulated in the associated deformations only. 
For sufficiently small time steps local existence and uniqueness are proved using a suitable fixed point formulation and the implicit function theorem.
A spatial Galerkin discretization with cubic splines on coarse meshes for the deformation
and piecewise bilinear finite elements on fine meshes for the image intensities are used to derive a fully practical algorithm.
Different applications underline the efficiency and stability of the proposed approach.\footnote{This paper is an extension of the prior proceedings paper \cite{EfRuSc17}.}
\end{abstract}

\section{Introduction}
Riemannian geometry has influenced imaging and computer vision tremendously in the past decades.
In particular, many methods in image processing have benefited from concepts emerging from Riemannian geometry
like geodesic curves, the logarithm, the exponential map, and parallel transport.
For example, when considering the space of images as an infinite-dimensional Riemannian manifold, the exponential map of an input image \wrt an initial variation
corresponds to an image extrapolation in the direction of this infinitesimal variation.
In particular, the large deformation diffeomorphic metric mapping (LDDMM) framework proved to be a powerful tool underpinned with the rigorous mathematical theory
of diffeomorphic flows. In fact, Dupuis \etal \cite{DuGrMi98} showed that the resulting flow is actually a flow of diffeomorphism.
Trouv\'e \cite{Tr95a,Tr98} exploited Lie group methods to construct a distance in the space of deformations.
In \cite{JoMi00}, Joshi and Miller applied this framework to (inexact and exact) landmark matching.
Beg \etal \cite{BeMiTr05} studied Euler--Lagrange equations for minimizing vector fields in the LDDMM framework and proposed an efficient algorithm
incorporating a gradient descent scheme and a semi-Lagrangian method to integrate the velocity fields.
Miller \etal \cite{MiTrYo06} proved the conservation of the initial momentum in Lagrangian coordinates associated with a geodesic in the LDDMM framework,
which allows for the stable computation of geodesic curves.
Younes \cite{Yo07} used Jacobi fields in the flow of diffeomorphism approach to derive gradient descent methods for the path energy.
In \cite{HaZaNi09}, Hart \etal exploited the optimal control perspective to the LDDMM model with the motion field as the underlying control.
Vialard \etal \cite{ViRiRu12a,ViRiRu12} studied methods from optimal control theory
to accurately estimate this initial momentum and to relate it to the Hamiltonian formulation of the geodesic flow. 
Furthermore, they used the associated Karcher mean to compute intrinsic means of medical images.
Vialard and Santambrogio investigated in \cite{ViSa09} the flow of diffeomorphism approach for images in the space of functions of bounded variation.
In particular, they were able to rigorously derive an Euler--Lagrange equation for the formulation with a matching energy.
Lorenzi and Pennec \cite{LoPe13} applied the LDDMM framework to compute geodesics and parallel transport 
using Lie group methods.

The metamorphosis model \cite{MiYo01,TrYo05} generalizes the flow of diffeomorphism approach allowing for intensity variations along transport paths
and associated a corresponding cost functional with these variations. 
In \cite{TrYo05a}, Trouv\'e and Younes rigorously analyzed the local geometry of the resulting Riemannian manifold
and proved the existence of geodesic curves for square-integrable images and the (local) existence as well as the uniqueness of solutions of the initial value
problem for the geodesic equation in the case of images with square-integrable weak derivatives.
Holm \etal \cite{HoTrYo09} studied a Lagrangian formulation for the metamorphosis model
and proved existence for both the boundary value and the initial value problem in the case of measure-valued images.
Hong \etal \cite{HoJoSa12} proposed a metamorphic regression model and developed a shooting method to reliably recover initial momenta.

A comprehensive overview of most of the aforementioned topics is given in the book by Younes \cite{Yo10}, for a historic account we additionally refer to \cite{MiTrYo02}.

In \cite{BeEf14}, a variational time discretization of the metamorphosis model based on a sequence of simple, elastic image matching problems 
was introduced and $\Gamma$-convergence to the time continuous metamorphosis model
was proven. Furthermore, using a finite element discretization in space a robust algorithm was derived.
Exploiting de Casteljau's algorithm, this approach could also be used to compute discrete Riemannian B\'ezier curves
in the space of images \cite{EfRuSi14}.
In \cite{BeBuEf17}, the geodesic interpolation proposed in \cite{BeEf14} was employed to analyze the temporal evolution of a macular degeneration for
medical images acquired with an optical coherence tomography device, where an efficient GPU implementation is used to speed up the registration subproblems.

In this paper, we focus on the discrete exponential map associated with the time discrete metamorphosis model.
The Euler--Lagrange equations of the discrete path energy proposed in \cite{BeEf14} give rise to a set of equations, which characterize time steps 
of a discrete initial value problem for a given initial image and a given initial image variation. We study this time stepping problem both analytically and numerically. 
We will prove existence and uniqueness of solutions of the single time step problem, 
which is guaranteed to generate time discrete geodesics in the sense of the time discrete variational approach. 
A straightforward treatment, for instance via a Newton scheme, would lead to higher order derivatives of image functions concatenated with diffeomorphic deformations,
which are both theoretically and numerically very difficult to treat. We show how to avoid these difficulties using a proper transformation of the defining Euler--Lagrange equations
and reduce the number of unknowns.
With respect to the existence, we apply a fixed point argument based on Banach's fixed point theorem for images bounded in $H^1$ and an initial variation, which is supposed to be small in $L^2$.
The uniqueness proof is based on an implicit function theorem argument for initial variations, which are small in $H^1$.
Finally, the numerical algorithm picks up the fixed point approach for another variant of the Euler--Lagrange equations.
\medskip

Compared to the proceedings paper \cite{EfRuSc17}, which introduced
the discrete exponential map in the context of the time discrete metamorphosis model and the
numerical optimization algorithm, we give in this paper the comprehensive derivation of the method,
formulate and prove the existence of the discrete exponential map.
Furthermore, two additional applications are presented.
\medskip

The paper is organized as follows: In Section \ref{sec:review}, we briefly recall the metamorphosis model in the time continuous and time discrete setting, respectively.
Departing from the Euler--Lagrange equations of a time discrete geodesic, a single time step of the discrete exponential map is derived in Section \ref{sec:expMap}.
Then the discrete geodesic shooting relies on the iterative application of the one step extrapolation.
In Section \ref{sec:ELE}, local existence and local uniqueness of this discrete exponential map are proven
based on a suitable combination of the implicit function theorem and Banach's fixed point theorem.
The fixed point formulation is also used in Section \ref{sec:algo} to derive an efficient and stable algorithm.
Finally, numerical results for different applications are presented in Section \ref{sec:results}.
\medskip

We use standard notation for Lebesgue and Sobolev spaces on the image domain $\domain$, \ie $L^p(\domain)$ and $H^m(\domain)=W^{m,2}(\domain)$.
The associated norms are denoted by $\|\cdot\|_{L^p(\domain)}$ and $\|\cdot\|_{H^m(\domain)}$, respectively, and the seminorm in $H^m(\domain)$ is given by $|\cdot|_{H^m(\domain)}$.
Furthermore, $H^m_0(\domain)$ is the closure of $C^\infty(\domain)$ functions with compact support \wrt the norm $\|\cdot\|_{H^m(\domain)}$ and its dual is denoted by 
$H^{-m}(\domain)$.
For any $f,g\in H^m(\domain)$, $m\geq 1$, we set
\[
D^m f\cdot D^m g=\sum_{i_1,\ldots,i_m=1}^n\frac{\partial^m f}{\partial_{x_{i_1}}\cdots\partial_{x_{i_m}}}\cdot\frac{\partial^m g}{\partial_{x_{i_1}}\cdots\partial_{x_{i_m}}}\,,
\qquad
\left|D^m f\right|=\left(D^m f\cdot D^m f\right)^{\frac{1}{2}}\,.
\]
The polyharmonic operator is inductively defined by $\Delta^m f:=\Delta(\Delta^{m-1}f)$ for $f\in H^{2m}(\domain)$ with $m\geq 2$.
Depending on the context, $\Id$ denotes either the identity mapping or the identity matrix.
For a matrix $A\in\R^{n,n}$, we refer to $A^\sym=\frac{1}{2}(A+A^T)$ as the symmetric part of $A$.
The symbol ``:'' indicates the sum over all pairwise products of two tensors.
Finally, we denote the variational derivative of a functional $J$ at a point $A$ in a direction $B$ by $\partial_A J[A](B)=\frac{\mathrm{d}}{\mathrm{d}\epsilon}J[A+\epsilon B]\big|_{\epsilon=0}$.

\section{Review of the metamorphosis model and its time discretization}\label{sec:review}
In this section, we briefly recall in a non-rigorous fashion the Riemannian geometry of the space of images based on the flow of diffeomorphism and its extension, the metamorphosis model.
For a detailed exposition of these models we refer to \cite{DuGrMi98,MiYo01,TrYo05,HoTrYo09,TrYo05a}.

Throughout this paper, we suppose that the image domain $\domain\subset\R^n$ for $n\in\{2,3\}$ has Lipschitz boundary.
For a flow of diffeomorphisms
$\deformation(t):\overline\domain\rightarrow\R^n$ for $t\in [0,1]$ driven by the
\emph{Eulerian velocity} $v(t)=\dot\deformation(t)\circ\deformation^{-1}(t)$ we take into account a quadratic form $L$ subjected to certain growth and consistency conditions,
which can be considered as a Riemannian metric on the space of diffeomorphisms and thus on the space of diffeomorphic transformations $\image(t)=\image_0\circ\deformation^{-1}(t)$
of a given reference image $\image_0$.
Based on these ingredients one can define the associated \emph{(continuous) path energy}
\[
\widetilde\pathenergy[(\deformation(t))_{t\in [0,1]}]=\int^1_0\int_\domain L[v,v] \d x \d t\,.
\]
By construction, this model comes with the brightness constancy assumption in the sense
that the material derivative $\frac{D}{\partial t}\image=\dot\image+v\cdot\nabla\image$ vanishes along the motion paths.
Contrary to this, the metamorphosis approach allows for image intensity variations along motion paths and
penalizes the integral over the squared material derivative as an additional term in the metric.
Hence, the \emph{path energy} in the metamorphosis model for an image curve $\image\in L^2((0,1),L^2(\domain))$ and $\delta>0$ is defined as
\begin{equation}
\pathenergy[\image]:=\int_0^1\inf_{(v,z)}\int_\domain L[v,v]+\frac1\delta z^2 \d x \d t\,,
\label{eq:metamorphosisContModel}
\end{equation}
where the infimum is taken over all pairs $(v,z)$ which fulfill the transport equation $\frac{D}{\partial t}\image=\dot\image+v\cdot\nabla\image= z$.
Here, we consider
\[
L[v,v]:=Dv:Dv+\gamma\Delta^m v\cdot\Delta^m v
\]
with $\gamma>0$ and $2m>1+\frac{n}{2}$.
To formulate this rigorously, one has to take into account the weak material derivative $z\in L^2((0,1),L^2(\domain))$ defined via 
the equation 
\[
\int_0^1 \int_\domain\eta z\d x \d t=-\int_0^1\int_\domain(\partial_t\eta+\div(v\eta))\image\d x \d t
\]
for all $\eta\in C^{\infty}_c((0,1)\times\domain)$.
Geodesic curves are defined as minimizers of the path energy \eqref{eq:metamorphosisContModel}.
Under suitable assumptions, one can prove the existence of a \emph{geodesic curve} in the class of all regular curves with prescribed initial and end image.
For the definition of regular curves and the existence proof we refer to \cite{TrYo05a}.
\medskip

In what follows, we consider the time discretization of the path energy \eqref{eq:metamorphosisContModel}
proposed in \cite{BeEf14} adapted to the slightly simpler transport cost $L[\,\cdot\,,\,\cdot\,]$.
To this end, we define for arbitrary images 
$\image,\tilde\image\in L^2(\domain)$ the \emph{discrete matching energy} 
\begin{equation}
\energy[\image,\tilde \image]:=
\min_{\deformation\in\admset}
\left\{
\energy^{D}[\image,\tilde \image,\deformation]:=
\int_\domain|D\deformation-\Id|^2+\gamma|\Delta^m\deformation|^2+\frac{1}{\delta}(\tilde\image\circ\deformation-\image)^2\d x\right\}\,,
\label{eq:WEnerDefinition}
\end{equation}
which is composed of a rescaled thin plate regularization term (first two terms) and a quadratic $L^2(\domain)$-mismatch measure (\cf~\cite[(6.2)]{BeEf14}).
The \emph{set of admissible deformations} $\admset$ is defined as
\begin{equation*}
\admset:=\left\{\deformation\in H^{2m}(\domain,\domain):\deformation-\Id\in H_0^{2m}(\domain,\domain)\right\}\,.
\end{equation*}
\begin{remark}
In \cite{BeEf14}, for every admissible deformation $\deformation$ the weaker boundary condition $\deformation=\Id$ on $\partial\domain$
instead of $\deformation-\Id\in H_0^{2m}(\domain,\domain)$ was assumed. Here, this stronger condition is required for both a higher regularity result (\cf Proposition~\ref{prop:bdryRegularity}) 
and a higher order control of the deformations (\cf \eqref{eq:boundphi}). With these altered boundary conditions the equality
\begin{equation}
\int_\domain|\Delta^m\testDeformation|^2\d x=\int_\domain|D^{2m}\testDeformation|^2\d x\qquad\text{for all }\testDeformation\in H_0^{2m}(\domain,\domain)
\label{eq:normEquality}
\end{equation}
holds true for all $m\geq 1$ (\cf \cite[Section 2.2]{GaGrSw10}). In fact, using integration by parts we exemplarily obtain for $m=1$
\[
\int_\domain|\Delta\testDeformation|^2\d x=
\int_\domain\sum_{i,j=1}^n\partial_i^2\testDeformation\cdot\partial_j^2\testDeformation\d x=
\int_\domain\sum_{i,j=1}^n\partial_i\partial_j\testDeformation\cdot\partial_i\partial_j\testDeformation\d x=
\int_\domain|D^2\testDeformation|^2\d x\,.
\]
\end{remark}
In what follows, we need the existence of minimizers of this particular matching energy $\energy$ for input images $\image,\tilde\image\in L^2(\domain)$
with $\|\tilde\image-\image\|_{L^2(\domain)}$ sufficiently small.
\begin{proposition}[Existence of a minimizing deformation for $\energy$]\label{prop:existenceDeformation}
Let $\image\in L^2(\domain)$ and $2m-\frac{n}{2}>1$. Then there exists a constant $C_\energy>0$ that solely depends on $\gamma$, $\delta$, $\domain$ and $m$ such that for every
$\tilde\image\in L^2(\domain)$ with $\|\tilde\image-\image\|_{L^2(\domain)}\leq C_\energy$ there is a minimizing deformation $\deformation\in\admset$ for $\energy$, \ie
$\energy[\image,\tilde \image]=\energy^{D}[\image,\tilde \image,\deformation]$, and $\deformation$ is a $C^1(\domain)$-diffeomorphism.
\end{proposition}
\begin{proof}
The proof is based on the direct method in the calculus of variations. Let $(\deformation^j)_{j\in\N}\subset\admset$ be a minimizing sequence for $\energy^D[\image,\tilde\image,\deformation^j]$
with monotonously decreasing energy and
\begin{equation}
0\leq\inf_{\tilde\deformation\in\admset}\energy^D[\image,\tilde\image,\tilde\deformation]=\lim_{j\rightarrow\infty}\energy^D[\image,\tilde\image,\deformation^j]
\leq\overline{\mathbf{W}}:=\energy^D[\image,\tilde\image,\Id]=\tfrac{1}{\delta}\|\tilde\image-\image\|_{L^2(\domain)}^2\,.
\label{eq:estimatesMinimizing}
\end{equation}
Since $\deformation^j-\Id\in\testSpace$, \eqref{eq:normEquality} and \eqref{eq:estimatesMinimizing} imply
\[
\int_\domain\gamma|D^{2m}(\deformation^j-\Id)|^2\d x=
\int_\domain\gamma|\Delta^m(\deformation^j-\Id)|^2\d x=
\int_\domain \gamma|\Delta^m\deformation^j|^2\d x\leq\frac{1}{\delta}\|\tilde\image-\image\|_{L^2(\domain)}^2\leq\frac{C_\energy^2}{\delta}\,.
\]
Thus, the norm equivalence of $\|\cdot\|_{H^{2m}(\domain)}$
and $|\cdot|_{H^{2m}(\domain)}$ for the space $\testSpace$, which follows by
an iterative application of the Poincar\'e inequality (\cf \cite[Corollary 6.31]{AdFo03}), yields
\begin{equation}
\|\deformation^j-\Id\|_{H^{2m}(\domain)}\leq C\|\tilde\image-\image\|_{L^2(\domain)}\leq CC_\energy\,.
\label{eq:controlDeformation}
\end{equation}
By taking into account the embedding $H^{2m}(\domain)\hookrightarrow C^1(\overline\domain)$ and considering a smaller $C_\energy$ if necessary
we can assume 
\[
\|\det(D\deformation^j)-1\|_{L^\infty(\domain)}\leq C_d
\]
for a constant $C_d\in(0,1)$,
which implies that $\deformation^j$ is $C^1(\domain)$-diffeomorphism (see \cite[Theorem 5.5-2]{Ci88}).
Moreover, since $(\deformation^j)_{j\in\N}$ are uniformly bounded in $H^{2m}(\domain)$ (\cf \eqref{eq:controlDeformation}),
a subsequence (also denoted by $\deformation^j$) converges weakly in $H^{2m}(\domain)$ due to the reflexivity of
this space and (strongly) in $C^{1,\alpha}(\overline\domain)$ for $\alpha\in(0,2m-1-\frac{n}{2})$ to a $C^1(\domain)$-diffeomorphism $\deformation\in\admset$.

Next, we prove the convergence of the $L^2(\domain)$-mismatch terms. To this end, we estimate
\begin{align*}
&\quad \left|\int_\domain|\tilde\image\circ\deformation^j-\image|^2-|\tilde\image\circ\deformation-\image|^2\d x\right|
\leq \int_\domain(|\tilde\image\circ \deformation^j-\image|+|\tilde\image\circ \deformation-\image|)|\tilde\image\circ\deformation^j-\tilde\image\circ\deformation|\d x \\
&\leq\left(\|\tilde\image\circ\deformation^j-\image\|_{L^2(\domain)}+\|\tilde\image\circ\deformation-\image\|_{L^2(\domain)}\right)\|\tilde\image\circ\deformation^j-\tilde\image\circ\deformation\|_{L^2(\domain)}
\leq 2\sqrt{\delta\overline{\mathbf{W}}}\|\tilde\image\circ\deformation^j-\tilde\image\circ\deformation\|_{L^2(\domain)}\,.
\end{align*}
Now, we approximate $\tilde\image$ in $L^2(\domain)$ by a sequence of smooth functions $(\tilde\image_i)_{i\in \N}$
such that $\|\tilde\image-\tilde \image_i\|_{L^2(\domain)}\leq 2^{-i}$. Then,
\begin{equation}
\|\tilde\image\circ\deformation^j-\tilde\image\circ\deformation\|_{L^2(\domain)}
\leq\|\tilde\image\circ\deformation^j-\tilde \image_i\circ\deformation^j\|_{L^2(\domain)}
+\|\tilde\image_i \circ\deformation^j-\tilde \image_i\circ \deformation\|_{L^2(\domain)}
+\|\tilde\image_i \circ\deformation-\tilde\image\circ \deformation\|_{L^2(\domain)}\,.
\label{eq:sumDecomposition}
\end{equation}
Next, applying the transformation formula yields
\[
\|\tilde\image\circ\deformation^j-\tilde \image_i\circ\deformation^j\|_{L^2(\domain)}
\leq\|(\det(D\deformation^j)\circ(\deformation^j)^{-1})^{-1}\|_{L^\infty(\domain)}^\frac{1}{2}\|\tilde\image-\tilde\image_i\|_{L^2(\domain)}
\leq\frac{1}{(1-C_d)^\frac{1}{2}}\|\tilde\image-\tilde\image_i\|_{L^2(\domain)}\,.
\]
Likewise, we can deduce $\|\tilde\image\circ\deformation-\tilde \image_i\circ\deformation\|_{L^2(\domain)}\leq C\|\tilde\image-\tilde\image_i\|_{L^2(\domain)}$.
Furthermore, the middle term in \eqref{eq:sumDecomposition} vanishes for fixed $i$ as $j\rightarrow\infty$. Finally,
using the lower semicontinuity of the first two terms of the energy we get
\[
\energy^D[\image,\tilde\image,\deformation]\leq\liminf_{j\rightarrow\infty}\energy^D[\image,\tilde\image,\deformation^j]\,,
\]
which proves this proposition.
\end{proof}
Following the general approach for the variational time discretization of geodesic calculus in \cite{RuWi12b} and the particular discretization of the metamorphosis model in \cite{BeEf14}, we 
define the discrete path energy $\Pathenergy_K$ on a sequence of $K+1$ images $(\image_0,\ldots,\image_K)\in(L^2(\domain))^{K+1}$  with $K\geq 2$ 
as the weighted sum of the discrete matching energy evaluated at consecutive images, \ie
\begin{equation}\label{eq:pathenergy}
\Pathenergy_K[\image_0,\ldots,\image_K]:=K\sum_{k=1}^K\energy[\image_{k-1},\image_k]\,.
\end{equation}
A $(K+1)$-tuple $(\image_0,\ldots,\image_K)\in(L^2(\domain))^{K+1}$ with given images $\image_0$ and $\image_K$ is defined to be a \emph{discrete geodesic curve} connecting 
$\image_0$ and $\image_K$ if
it minimizes $\Pathenergy_K$ \wrt all other $(K+1)$-tuples with $\image_0$ and $\image_K$  fixed.
For the proof of the existence of discrete geodesics we refer to \cite{BeEf14}.
It is also shown in \cite{BeEf14} that a suitable extension of the discrete path energy $\Pathenergy_K$
$\Gamma$-convergences to the continuous path energy $\pathenergy$.
Let us finally mention that neither the matching deformation in \eqref{eq:WEnerDefinition} nor the discrete geodesic curve defined as the minimizer of 
\eqref{eq:pathenergy} for given input images $\image_0$ and $\image_K$ are necessarily unique.

\section{The time discrete exponential map}\label{sec:expMap}
In this section, we define the discrete exponential map and derive optimality as well as regularity results,
on which the study of existence and uniqueness in Section \ref{sec:ELE} and the algorithm introduced in Section \ref{sec:algo} will be based.

Let us briefly recall the definition of the continuous exponential map on a Riemannian manifold.
Let $\y:[0,1]\rightarrow\manifold$ be the unique geodesic curve for a prescribed initial position $\y(0)=\y_A$ and an initial velocity $\dot{\y}(0)=v$ on a Riemannian manifold~$(\manifold,\metric)$.
The exponential map is then defined as $\exp_{\y_A}(v)=\y(1)$. Furthermore, one easily checks that $\exp_{\y_A}(\frac{k}{K}v)=\y(\frac{k}{K})$ for $0\leq k\leq K$. We refer to  the textbook \cite{Kl95a} for a detailed discussion of the (continuous) exponential map. 
Now, we ask for a time discrete counterpart of the exponential map in the metamorphosis model.
To this end, we consider an image $\image_0$ as the initial data and a second image $\image_1$ such that $\zeta_1 = \image_1-\image_0$ represents
a small variation of the image~$\image_0$. 
This variation $\zeta_1$ is the discrete counterpart of the infinitesimal variation given by the velocity $v$ in the continuous case.
For varying values of $K \geq 2$ we now ask for a discrete geodesic $(\image_0, \image_1, \image_2,\ldots, \image_K)$ described as the minimizer of the discrete path energy \eqref{eq:pathenergy}.
Let us for the time being suppose that this geodesic is unique -- a property to be verified later.
Based on our above observation for the continuous exponential map we define $\Exp{k}_{\ast}(\,\cdot\,)$ as the discrete counterpart of $\exp_{\ast}(\frac{k}{K}\,\cdot\,)$, \ie we set
\[
\Exp{k}_{\image_0}(\zeta_1):=\image_k
\]
for $k=1,\ldots,K$. The definition of the exponential map $\Exp{k}_{\image_0}(\zeta_1)$ does not depend on the number of time steps $K$. Indeed, if
$(\image_0,\image_1, \image_2,\ldots, \image_K)$ is a discrete geodesic, then $(\image_0, \image_1, \image_2,\ldots, \image_L)$ with $L\leq K$ is also a geodesic.
Taking into account $k=2$ we immediately observe that the sequence of discrete exponential maps $(\Exp{k}_{\image_0}(\zeta_1))_{k=1,\ldots}$  can iteratively be defined as follows
\begin{equation}
\Exp{k}_{\image_0}(\zeta_1)=\image_k := \Exp{2}_{\image_{k-2}}(\zeta_{k-1})
\label{eq:recursiveDefExp}
\end{equation}
for $k\geq 2$, where $\zeta_{k-1} = \image_{k-1}-\image_{k-2}$, and for the sake of completeness we define
$\Exp{0}_{\image_0}(\zeta_1)=\image_0$ and $\Exp{1}_{\image_0}(\zeta_1)=\image_1=\image_0+\zeta_1$.
Thus, it essentially remains to compute $\Exp{2}$ for a given input image $\image_{k-2}$ and an image variation $\zeta_{k-1}=\image_{k-1}-\image_{k-2}$ (see Figure \ref{fig:EXPscheme}).
For a detailed discussion of the discrete exponential map in the simpler model of Hilbert manifolds we refer to \cite{RuWi12b}.
The particular challenge here is that the matching energy $\energy$ cannot be evaluated directly, but requires to solve the variational problem \eqref{eq:WEnerDefinition}
for the matching deformation.
\begin{figure}[htb]
\begin{center}
\resizebox{0.9\linewidth}{!}{
\includegraphics{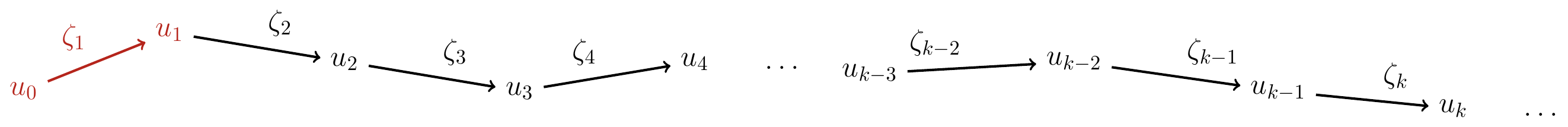}
}
\caption{Schematic drawing of $\Exp{k}_{\image_0}(\zeta_1)$, $k=1,\ldots,K$, the input data is highlighted in red.}
\label{fig:EXPscheme}
\end{center}
\end{figure}

\medskip

There are two major restrictions regarding the input images $\image_0$ and $\image_1$:
\medskip

Firstly,  the existence and uniqueness result for the discrete exponential map (\cf Section \ref{sec:ELE}) will require
weakly differentiable input images. These weak derivatives of images
naturally arise in the Euler--Lagrange equations for $\Exp{2}$ \wrt the deformations (see  \eqref{eq:ELEIImage2} and \eqref{eq:ELEI} below).
Let us remark that the weak differentiability of the input data for the exponential map is also a crucial requirement in the initial value
problem for the geodesic equation in \cite{TrYo05a}. Furthermore, the $H^1(\domain)$-regularity property is inherited along discrete geodesics, 
\ie $\Exp{k}_{\image_0}(\image_1-\image_0)\in H^1(\domain)$ for any $k\geq 1$
provided that $\image_0,\image_1\in H^1(\domain)$ (\cf \cite[Remark 3.3 and Equation (3.2)]{BeEf14}).
\medskip

Secondly, the initial variation $\zeta_1=\image_1-\image_0$ is assumed
to be sufficiently small in $L^2(\domain)$ in order to ensure the existence of the initial deformation $\deformation_1$
and guarantee the convergence of a suitable fixed point algorithm --
a property which appears to be natural in light of the analogue assumption for the continuous exponential map \cite{Kl95a}.
We will also see that for fixed $K$ the variations $\image_{k-1}-\image_{k-2}$ for $k\leq K$ will remain small provided that $\zeta_1$ is small.
Thus, for fixed $K$ the discrete exponential map $\Exp{k}_{\image_0}(\,\cdot\,)$ will be well-posed for a sufficiently small initial variation $\zeta_1$.
\medskip

Hence, in what follows we consider images in $H^1(\domain)$ and define $\image_k:=\Exp{2}_{\image_{k-2}}(\zeta_{k-1})$ 
as the (unique) image in $H^1(\domain)$ such that 
\begin{equation} 
\label{eq:DefExp2}
\image_{k-1}=\argmin_{\image\in H^1(\domain)}\min_{\deformation_{k-1},\deformation_k\in\admset}
\energy^{D}[\image_{k-2},\image,\deformation_{k-1}]+\energy^{D}[\image,\image_{k},\deformation_{k}]\,.
\end{equation}
For the sake of simplicity, we restrict to the first step in the iterative computation of the discrete exponential map with $k=2$.
Given $\image_0,\image_1 \in H^1(\domain)$ the first order optimality conditions for \eqref{eq:DefExp2} for $\image_2\in H^1(\domain)$ 
and $\deformation_1,\deformation_2\in\admset$ read as
\begin{align}
\begin{split}
\partial_{\image_1}(\energy^{D}[\image_0,\image_1,\deformation_1]+\energy^{D}[\image_1,\image_2,\deformation_2])(\testImage)&=0\,,\\
\partial_{\deformation_1}\energy^{D}[\image_0,\image_1,\deformation_1](\testDeformation)&=0\,,\\
\partial_{\deformation_2} \energy^{D}[\image_1,\image_2,\deformation_2](\testDeformation)&=0\,,
\end{split}
\label{eq:abstractOptimalityConditions}
\end{align}
for all $\testImage\in H^1(\domain)$ and all $\testDeformation\in\testSpace$ for $2m-\frac{n}{2}>2$.
The system \eqref{eq:abstractOptimalityConditions} is equivalent to
\begin{align}
\int_\domain(\image_1\circ\deformation_1-\image_0)\testImage\circ\deformation_1-(\image_2\circ\deformation_2-\image_1)\testImage\d x\label{eq:ELEIImage1}&=0\,,\\ 
\int_\domain 2D\deformation_1:D\testDeformation+2\gamma\Delta^m\deformation_1\cdot\Delta^m\testDeformation+\frac{2}{\delta}(\image_1\circ\deformation_1-\image_0)(\nabla\image_1\circ\deformation_1)\cdot\testDeformation\d x&=0\,,\label{eq:ELEIImage2}\\ 
\int_\domain 2D\deformation_2:D\testDeformation+2\gamma\Delta^m\deformation_2\cdot\Delta^m\testDeformation+\frac{2}{\delta}(\image_2\circ\deformation_2-\image_1)(\nabla\image_2\circ\deformation_2)\cdot\testDeformation\d x&=0\,.
\label{eq:ELEI}
\end{align}
The subsequent lemma provides a reformulation of the above system of equations, in which the dependency of the unknown function $\image_2$ in 
\eqref{eq:ELEI} is removed and in addition solely the function $\image_1$ and no longer derivatives of $\image_1$ appear.
\begin{lemma}[Reformulation of the Euler--Lagrange equation for $\deformation_2$]\label{Lemma:reform}
Let $\image_0,\image_1,\image_2\in H^1(\domain)$ such that (\cf Proposition~\ref{prop:existenceDeformation})
\[
\|\image_1-\image_0\|_{L^2(\domain)},\|\image_2-\image_1\|_{L^2(\domain)}\leq C_\energy\,,
\]
$2m-\frac{n}{2}>2$, and assume that \eqref{eq:ELEIImage1} and \eqref{eq:ELEIImage2} hold true.
Let $\deformation_i$ with $i=1,2$ be the minimizer of $\energy^D[\image_{i-1},\image_i,\,\cdot\,]$ on $\admset$ according to Proposition \ref{prop:existenceDeformation}.
\begin{enumerate}[label=(\roman*)]
\item\label{item:firstReformulation}
Then \eqref{eq:ELEI} is equivalent to
\begin{align}
\begin{split}
\int_\domain 2\gamma\Delta^m\deformation_2\cdot\Delta^m\testDeformation+2D\deformation_2:D\testDeformation+\frac{2}{\delta}(\image_1\circ\deformation_1-\image_0)(\nabla\image_1\cdot(D\deformation_2)^{-1}\testDeformation)\circ \deformation_1&\\
+\frac{1}{\delta} \frac{(\image_1\circ \deformation_1-\image_0)^2}{\det D\deformation_1} \left((D\deformation_2)^{-T}:(D^2\deformation_2(D\deformation_2)^{-1}\testDeformation)-(D\deformation_2)^{-T}:D\testDeformation\right)\circ\deformation_1\d x&=0
\end{split}
\label{eq:ELEII}
\end{align}
for all $\testDeformation\in\testSpace$.
\item\label{item:secondReformulation}
Under the additional assumptions that $\partial\domain\in C^{4m}$ and  $\image_0,\image_1,\image_2\in L^\infty(\domain)\cap H^1(\domain)$ the equation
\eqref{eq:ELEI} is equivalent to
\begin{align}
\begin{split}
\int_\domain& 2\gamma\Delta^m\deformation_2\cdot\Delta^m\testDeformation+2D\deformation_2:D\testDeformation\d x\\
=\int_\domain&2\gamma\Delta^m \deformation_1\cdot\Delta^m(((D\deformation_2)^{-1}\testDeformation)\circ\deformation_1)+
2D\deformation_1:D(((D\deformation_2)^{-1}\testDeformation)\circ\deformation_1)\\
&-\frac{1}{\delta} \frac{(\image_1\circ \deformation_1-\image_0)^2}{\det D\deformation_1}\left((D\deformation_2)^{-T}:(D^2\deformation_2(D\deformation_2)^{-1}\testDeformation)-(D\deformation_2)^{-T}:D\testDeformation\right)\circ\deformation_1\d x\,.
\end{split}
\label{eq:ELEIII}
\end{align}
Here, the notation
$(D^2\deformation_2(D\deformation_2)^{-1}\testDeformation)_{jk}=\sum_{i,l=1}^n\partial_j\partial_k\deformation^i_2(D\deformation_2)^{-1}_{il}\testDeformation_l$ is used.
\end{enumerate}
\end{lemma}
\begin{proof}
By using the transformation formula the energy $\energy^D$ can be rewritten as follows
\[
\energy^D[\image_1,\image_2,\deformation_2]=\int_\domain 
|D\deformation_2-\Id|^2+\gamma|\Delta^m\deformation_2|^2+\frac{1}{\delta}\frac{(\image_2-\image_1\circ\deformation_2^{-1})^2}{\det(D\deformation_2)\circ\deformation_2^{-1}}\d x\,,
\]
since $\deformation_2\in\admset$ is a diffeomorphism (see Proposition~\ref{prop:existenceDeformation}).
As a next step, we rewrite the Euler--Lagrange equation \wrt $\deformation_2$ of $\energy^D[\image_1,\image_2,\,\cdot\,]$. 
To this end, we use the identities $\partial_{\deformation_2}\deformation_2^{-1}(\testDeformation)=-((D\deformation_2)^{-1}\testDeformation)\circ\deformation_2^{-1}$,
which follows by differentiating $(\deformation_2+\epsilon\testDeformation)\circ(\deformation_2+\epsilon\testDeformation)^{-1}=\Id$ \wrt $\epsilon$,
and $\partial_A\det(A)(B)=\cof(A):B$ for $A\in GL(n)$ and $B\in\R^{n,n}$ with $\cof A = (\det A)A^{-T}$.
Thus, we obtain
\begin{align*}
\int_\domain 2D\deformation_2:D\testDeformation+2\gamma\Delta^m\deformation_2\cdot\Delta^m\testDeformation
+\frac{2}{\delta}(\image_2-\image_1\circ\deformation_2^{-1})\frac{(\nabla\image_1\cdot(D\deformation_2)^{-1}\testDeformation)\circ\deformation^{-1}_2}{\det(D\deformation_2)\circ\deformation^{-1}_2}& \\
+\frac{1}{\delta}\frac{(\image_2-\image_1\circ\deformation_2^{-1})^2}{(\det D\deformation_2)^2\circ \deformation_2^{-1}}\left(\cof D\deformation_2 : (D^2\deformation_2(D\deformation_2)^{-1}\testDeformation)-\cof D\deformation_2:D\testDeformation\right)\circ\deformation_2^{-1}\d x&=0\,.
\end{align*}
A further application of the transformation formula \wrt $\deformation_2$ yields
\begin{align}
\int_\domain 2D\deformation_2:D\testDeformation+2\gamma\Delta^m\deformation_2\cdot\Delta^m \testDeformation+
\frac{2}{\delta}(\image_2\circ\deformation_2-\image_1)\nabla\image_1\cdot(D\deformation_2)^{-1}\testDeformation&\notag\\
+\frac{1}{\delta}\frac{(\image_2\circ\deformation_2-\image_1)^2}{\det D\deformation_2}
\left(\cof D\deformation_2 :(D^2\deformation_2(D\deformation_2)^{-1}\testDeformation)-\cof D\deformation_2 : D\testDeformation\right)\d x&=0\,.
\label{eq:longExpressionDeformation}
\end{align}
To remove the dependency of the function $\image_2$ above, we employ the pointwise condition
\begin{equation}
\image_2\circ\deformation_2-\image_1=\frac{\image_1-\image_0\circ\deformation_1^{-1}}{\det(D\deformation_1)\circ\deformation_1^{-1}}
\label{eq:pointwiseImageUpdate}
\end{equation}
for a.e. $x\in\domain$, which follows directly from \eqref{eq:ELEIImage1}. Inserting this in \eqref{eq:longExpressionDeformation} and using the integral transformation formula we achieve 
\begin{align*}
\int_\domain 2D\deformation_2:D\testDeformation+2\gamma\Delta^m\deformation_2\cdot\Delta^m\testDeformation
+\frac{2}{\delta}(\image_1\circ\deformation_1-\image_0)(\nabla\image_1\cdot(D\deformation_2)^{-1}\testDeformation)\circ\deformation_1&\\
+\frac{1}{\delta}\frac{(\image_1\circ \deformation_1-\image_0)^2}{\det D\deformation_1}\left(\frac{\cof D\deformation_2:(D^2\deformation_2(D\deformation_2)^{-1}\testDeformation)-\cof D\deformation_2 : D\testDeformation}{\det D\deformation_2}\right)\circ\deformation_1\d x&=0\,. 
\end{align*}
The identity $\cof(A)=\det(A)A^{-T}$ for $A\in GL(n)$ implies \ref{item:firstReformulation}.

To show \ref{item:secondReformulation}, we take into account the test function $\testDeformationTwo:=((D\deformation_2)^{-1}\testDeformation)\circ\deformation_1$ in
\eqref{eq:ELEIImage2}. To justify this, we have to show that $\testDeformationTwo\in\testSpace$. To this end,
we require $H^{2m+1}(\domain)$-regularity of $\deformation_2$, which will follow from Proposition \ref{prop:bdryRegularity},
and classical differential calculus for Sobolev functions \cite{AdFo03}.
Inserting $\testDeformationTwo$ into \eqref{eq:ELEIImage2} we get
\[
-\int_\domain\frac{2}{\delta}(\image_1\circ\deformation_1-\image_0)(\nabla\image_1\cdot(D\deformation_2)^{-1}\testDeformation)\circ\deformation_1\d x
=\int_\domain 2\gamma\Delta^m\deformation_1\cdot\Delta^m(((D\deformation_2)^{-1}\testDeformation)\circ\deformation_1)+2D\deformation_1:D(((D\deformation_2)^{-1}\testDeformation)\circ\deformation_1)\d x\,.
\]
By adding the above equation to \eqref{eq:ELEII} we have proven \ref{item:secondReformulation}.
\end{proof}
\begin{proposition}[Maximal regularity of the deformations]\label{prop:bdryRegularity}
Let $2m-\frac{n}{2}>2$ and $\partial\domain\in C^{4m}$. Furthermore, let $\image_0,\image_1,\image_2\in L^\infty(\domain)\cap H^1(\domain)$
and suppose that $\deformation_1,\deformation_2\in\admset$ are minimizers of $\energy^D[\image_0,\image_1,\,\cdot\,]$ and $\energy^D[\image_1,\image_2,\,\cdot\,]$, respectively.
Then $\deformation_1,\deformation_2\in\admset\cap H^{4m}(\domain)$.
\end{proposition}
\begin{proof}

We only prove the result for $\deformation_2$, for $\deformation_1$ one proceeds analogously.
Let  $w$ be the displacement associated with $\deformation_2$, \ie $w=\deformation_2-\Id\in\testSpace$.
Using integration by parts in \eqref{eq:ELEI} we obtain for a test function  $\testDeformation\in\testSpace$
\begin{align*}
&\quad\int_{\domain}\Delta^m w\cdot\Delta^m\testDeformation\d x
=-\int_{\domain}\tfrac{1}{\gamma\delta}(\image_2\circ\deformation_2-\image_1)((\nabla\image_2\circ\deformation_2)\cdot\testDeformation)+\tfrac{1}{\gamma}D\deformation_2:D\testDeformation\d x\notag\\
&=-\int_{\domain}\tfrac{1}{\gamma\delta}(\image_2\circ\deformation_2-\image_1)((\nabla\image_2\circ\deformation_2)\cdot\testDeformation)-\tfrac{1}{\gamma}\Delta\deformation_2\cdot\testDeformation\d x
=:\int_{\domain}f\cdot\testDeformation\d x
\end{align*}
with $f\in L^2(\domain,\R^n)$. 
Then, the assertion follows from the general $L^2$-regularity theory for polyharmonic equations as presented in \cite[Section 2.5.2]{GaGrSw10}.
\end{proof}

\begin{remark}
Since $\deformation_2$ is a diffeomorphism, \eqref{eq:pointwiseImageUpdate} is equivalent to
\begin{equation}
\image_2=\left(\frac{\image_1-\image_0\circ\deformation_1^{-1}}{\det(D\deformation_1)\circ\deformation_1^{-1}}\right)\circ\deformation_2^{-1}+\image_1\circ\deformation_2^{-1}\,.
\label{eq:pointwiseImageUpdatealternative}
\end{equation}
Here, the first summand reflects the intensity modulation along the geodesic, the second summand
quantifies the contribution due to the transport.
\end{remark}
We will use the first reformulation \eqref{eq:ELEII} (Lemma~\ref{Lemma:reform}~\ref{item:firstReformulation}) of the Euler--Lagrange equation \eqref{eq:ELEI} with respect to $\deformation_2$ to derive a fixed point iteration in the existence proof 
for the time discrete exponential map. The second reformulation \eqref{eq:ELEIII} (Lemma~\ref{Lemma:reform}~\ref{item:secondReformulation})
will later be used in a modified and spatially discrete fixed point iteration in the numerical algorithm.

\section{Local existence and uniqueness of the discrete exponential map}\label{sec:ELE}
In this section, we prove local existence and local uniqueness for the discrete exponential map.
At first, we make use of an argument based on Banach's fixed point theorem applied to the reformulation of the Euler--Lagrange equation given in Lemma~\ref{Lemma:reform}~\ref{item:firstReformulation}
for a discrete geodesic $(\image_0,\image_1,\image_2)$  with deformations $\deformation_1$ and $\deformation_2$.
For image pairs $(\image_0,\image_1)$ we establish the existence of a solution $(\image_2,\deformation_1,\deformation_2)$
to the system of equations \eqref{eq:ELEIImage1}, \eqref{eq:ELEIImage2} and \eqref{eq:ELEI}
provided that $\image_0$ and $\image_1$ are bounded in $H^1(\domain)$ and close in $L^2(\domain)$.
This does not necessarily imply that for given $\image_0$ and $\image_1$ the resulting discrete path $(\image_0,\image_1,\image_2)$ is the unique discrete geodesic
connecting $\image_0$ and $\image_2$.
Thus, in a second step we will show that this indeed holds true if the images $\image_0$ and $\image_1$ are close in $H^1(\domain)$.
To this end, we apply an implicit function theorem argument (\cf the corresponding proof for the discrete 
exponential map on Hilbert manifolds given in \cite{RuWi12b}).
Let us remark that this argument also allows to establish existence, but under the stronger assumption that the input images are close in $H^1(\domain)$ compared to
the requirement of closeness in $L^2(\domain)$ and boundedness in $H^1(\domain)$ for the existence proof via the fixed point theorem.
Furthermore, the fixed point approach will be taken into account for the numerical approximation of the time discrete exponential map.

\begin{theorem}[Existence of solutions of the Euler--Lagrange equations]\label{thm:existenceELE}
Let $2m-\frac{n}{2}>2$, $\image_0\in H^1(\domain)$. Then there are constants $C_\image, c_\image > 0$ such that
for every 
$\image_1 \in \left\{\image\in H^1(\domain):\ |\image|_{H^1(\domain)} \leq C_\image,\; \|\image-\image_0\|_{L^2(\domain)}\leq c_\image\right\}$
there exists a solution $(\image_2,\deformation_1,\deformation_2)\in H^1(\domain)\times\admset\times\admset$
of \eqref{eq:ELEIImage1}, \eqref{eq:ELEIImage2} and \eqref{eq:ELEI}. In particular,
the defining system of equations for $\image_2=\Exp{2}_{\image_{0}}(\image_{1}-\image_{0})$ is solved.
\end{theorem}
\begin{proof} 
We begin with some preparatory considerations.
Let $c_\image\leq C_\energy$ and $\deformation_1\in\argmin_{\deformation\in\admset}\energy^{D}[\image_0,\image_1,\deformation]$ be a minimizing deformation (\cf Proposition \ref{prop:existenceDeformation}).
Following the same line of arguments as for the estimate \eqref{eq:controlDeformation} in the proof of Proposition~\ref{prop:existenceDeformation} we obtain
\begin{equation}
\|\deformation_1-\Id\|_{H^{2m}(\domain)}\leq C\|\image_1-\image_0\|_{L^2(\domain)}\leq Cc_\image\,.
\label{eq:boundphi}
\end{equation}
Furthermore, taking into account 
$\energy^{D}[\image_0,\image_1,\deformation_1]\leq\energy^{D}[\image_0,\image_1,\Id]$ we infer
\begin{equation}
\|\image_1\circ\deformation_1-\image_0\|_{L^2(\domain)}\leq\sqrt{\delta\,\energy^{D}[\image_0,\image_1,\deformation_1]}
\leq\sqrt{\delta\,\energy^{D}[\image_0,\image_1,\Id]}=\|\image_1-\image_0\|_{L^2(\domain)}\leq c_\image\,.
\label{eq:uniformL2}
\end{equation}

\medskip

\noindent
Now, we define the fixed point iteration and prove the contraction property in several steps:
\medskip

\noindent\textit{(i) Defining the fixed point mapping $\mathcal{F}$.}
Using Lemma~\ref{Lemma:reform}~\ref{item:firstReformulation} we define for a fixed deformation $\deformation_1$ the operators
$\operator,\secondOperator:\admset\to\testSpaceDual$ as
\begin{align*}
\operator[\deformation](\testDeformation)=
\int_\domain &-\frac{2}{\delta}(\image_1\circ\deformation_1-\image_0)(\nabla\image_1\cdot(D\deformation)^{-1}\testDeformation)\circ\deformation_1\notag\\
&-\frac{1}{\delta}\frac{(\image_1\circ\deformation_1-\image_0)^2}{\det D\deformation_1 } \left((D\deformation)^{-T}:(D^2\deformation(D\deformation)^{-1}\testDeformation)-(D\deformation)^{-T}:D\testDeformation\right)\circ\deformation_1\d x\,,\\
\secondOperator[\deformation](\testDeformation)=
\int_\domain & 2\gamma\Delta^m\deformation\cdot\Delta^m\testDeformation+2D\deformation:D\testDeformation\d x\notag
\end{align*}
for a diffeomorphism $\deformation\in\admset$ and all $\testDeformation\in\testSpace$, which allows us to reformulate the Euler--Lagrange equation \wrt the deformation $\deformation_2$ 
in \eqref{eq:ELEII} as 
\[
\operator[\deformation_2](\testDeformation)=\secondOperator[\deformation_2](\testDeformation)\,.
\]
Next, we will study the invertibility of the linear operator $\secondOperator$ and the Lipschitz continuity of $\operator$ with a Lipschitz constant
which depends monotonically on $c_\image$ and vanishes for $c_\image\searrow 0$.
This will imply that $\mathcal{F} := \secondOperator^{-1} \circ \operator$ is a contraction for sufficiently small $c_\image$.
The fixed point iteration to compute the unknown deformation~$\deformation_2$ reads as $\deformation^{j+1}=\mathcal{F}[\deformation^j]$ for $j\in\N$ and $\deformation^0=\Id$.

\medskip

\noindent\textit{(ii) Lipschitz continuity of $\operator$.}
In what follows, we assume that 
\[
\deformation,\tilde\deformation\in B_\epsilon(\Id):=\left\{\deformation:\deformation-\Id\in\testSpace,\|\deformation-\Id\|_{H^{2m}(\domain)}<\epsilon\right\}
\]
for a sufficiently small $\epsilon>0$, the dependency of $\epsilon$ on $C_\image$ and $c_\image$ is discussed below.
By the embedding $H^{2m}(\domain)\hookrightarrow C^2(\overline\domain)$ and for $\epsilon$ sufficiently small 
we may assume that
\begin{equation}
\|D\deformation-\Id\|_{L^\infty(\domain)}<\frac{1}{2}\,,\qquad \|\det(D\deformation)-1\|_{L^\infty(\domain)}<\frac{1}{2}
\label{eq:controlDerivatives}
\end{equation}
for all deformations $\deformation$ considered.
Since $\det(D\deformation(x))\geq\frac{1}{2}$ for all $x\in\domain$,
this ensures that such deformations are in $\admset$ and $C^1(\domain)$-diffeomorphisms (see \cite[Theorem 5.5-2]{Ci88}). Furthermore, we obtain
\[
\|\det(D\deformation^{-1})\|_{L^\infty(\domain)}\leq\Big(1-\|\det(D\deformation)-1\|_{L^\infty(\domain)}\Big)^{-1}<2\,,\ 
\|\cof(D\deformation)\|_{L^\infty(\domain)}\leq C\,,\ 
\|(D\deformation)^{-1}\|_{L^\infty(\domain)}\leq C\,,
\]
and deduce from  $(D\deformation)^{-1}=(\det(D\deformation))^{-1}\cof(D\deformation)^{T}$ 
\begin{align}
&\quad\|(D\deformation)^{-1}-(D\tilde{\deformation})^{-1}\|_{L^\infty(\domain)}
=\|(\det(D\deformation))^{-1}\cof(D\deformation)^{T}-(\det(D\tilde\deformation))^{-1}\cof(D\tilde\deformation)^{T}\|_{L^\infty(\domain)}\notag\\
&\leq
\left\|\tfrac{(\cof(D\deformation))^T}{\det(D\deformation)\det(D\tilde\deformation)}\right\|_{L^\infty(\domain)}
\|\det(D\deformation)-\det(D\tilde\deformation)\|_{L^\infty(\domain)}
+\|(\det(D\tilde\deformation))^{-1}\|_{L^\infty(\domain)}
\|\cof(D\deformation)^T-\cof(D\tilde\deformation)^T\|_{L^\infty(\domain)}\notag\\
&\leq C \|\deformation-\tilde\deformation\|_{H^{2m}(\domain)}
\label{eq:estimatesInverse}
\end{align}
for deformations $\deformation,\tilde\deformation\in B_\epsilon(\Id)$.
Thus, using the Cauchy-Schwarz inequality, the transformation formula, \eqref{eq:uniformL2}, \eqref{eq:controlDerivatives} and \eqref{eq:estimatesInverse}
we achieve the following estimate corresponding to the first term of $\operator$:
\begin{align*}
&\quad\Big|\int_\domain(\image_1\circ\deformation_1-\image_0)(\nabla\image_1\circ\deformation_1)\cdot((D\deformation)^{-1}\testDeformation)\circ\deformation_1
-(\image_1\circ\deformation_1-\image_0)(\nabla\image_1\circ\deformation_1)\cdot((D\tilde{\deformation})^{-1}\testDeformation)\circ\deformation_1\d x\Big| \\
&\leq  C \|\image_1\circ\deformation_1-\image_0\|_{L^2(\domain)}|\image_1|_{H^1(\domain)}
\|\det D(\deformation_1^{-1})\|_{L^\infty(\domain)}
\|(D\deformation)^{-1}-(D\tilde{\deformation})^{-1}\|_{L^\infty(\domain)} \|\testDeformation\|_{L^\infty(\domain)}\\
&\leq C C_\image c_\image \|\deformation-\tilde{\deformation}\|_{H^{2m}(\domain)}\|\testDeformation\|_{H^{2m}(\domain)}\,.
\end{align*}
Likewise, for the second term of $\operator$ we obtain by the transformation formula and by the embedding $H^{2m}(\domain)\hookrightarrow C^2(\overline\domain)$
\begin{align*}
&\quad\Big|\int_\domain\frac{(\image_1\circ\deformation_1-\image_0)^2}{\det D\deformation_1} 
\Big((D\deformation)^{-T}:(D^2\deformation(D\deformation)^{-1}\testDeformation)  -(D\tilde{\deformation})^{-T}:(D^2\tilde{\deformation}(D\tilde{\deformation})^{-1}\testDeformation)\\[-1ex]
&\hspace{10em}-(D\deformation)^{-T}:D\testDeformation
+(D\tilde{\deformation})^{-T}:D\testDeformation
\Big)\circ\deformation_1\d x\Big|\\
&\leq 
C c_\image^2 \|(\det D\deformation_1)^{-1}\|_{L^\infty(\domain)}
\Big(\left\|(D\deformation)^{-T}:(D^2\deformation(D\deformation)^{-1}\testDeformation)
-(D\tilde{\deformation})^{-T}:(D^2\tilde{\deformation}(D\tilde{\deformation})^{-1}\testDeformation)\right\|_{L^\infty(\domain)}\\
&\hspace{13em}+\left\|(D\deformation)^{-T}:D\testDeformation
-(D\tilde{\deformation})^{-T}:D\testDeformation\right\|_{L^\infty(\domain)}\Big)\\
&\leq
C c_\image^2 \|\deformation-\tilde\deformation\|_{H^{2m}(\domain)}\|\testDeformation\|_{H^{2m}(\domain)}\,.
\end{align*}
To conclude, for $\|\image_1-\image_0\|_{L^2(\domain)}\leq c_\image$ and $|\image_1|_{H^1(\domain)}\leq C_\image$ the mapping $\operator$ is indeed Lipschitz continuous on $B_{\epsilon}(\Id)\subset\admset$
and the Lipschitz constant is bounded by 
$C(C_\image c_\image+c_\image^2)$.
\medskip

\noindent\textit{(iii) Invertibility of  $\secondOperator$.}
The bilinear form
\[
\widetilde\secondOperator:\testSpace\times\testSpace\rightarrow\R\,,\quad
(\zeta,\testDeformation)\mapsto\int_\domain 2\gamma\Delta^m\zeta\cdot\Delta^m\testDeformation+2D\zeta:D\testDeformation\d x
\]
is bounded in $\testSpace$.
Furthermore, $\widetilde\secondOperator$ is coercive since for any $\testDeformation\in\testSpace$ we obtain
\[
\|\testDeformation\|_{H^{2m}(\domain)}^2\leq
C|\testDeformation|_{H^{2m}(\domain)}^2=
C\int_\domain \Delta^m\testDeformation\cdot\Delta^m\testDeformation\d x
\]
due to \eqref{eq:normEquality} and the iterative application of the Poincar\'e inequality (\cf \cite[Corollary 6.31]{AdFo03}).
Hence, by the Lax-Milgram Theorem (\cf \cite{GiTr92}) there exists for each $z\in\testSpaceDual$ a unique $\zeta\in\testSpace$ such that
$\widetilde\secondOperator[\zeta](\testDeformation)=z(\testDeformation)$ and $\widetilde\secondOperator^{-1}:\testSpaceDual\rightarrow\testSpace$
is a bounded operator. Finally, since $\secondOperator[\deformation]=\widetilde\secondOperator[\deformation-\Id]$ we can infer that $\secondOperator$
is a bounded and invertible operator with inverse $\secondOperator^{-1}[z]=\Id+\widetilde\secondOperator^{-1}[z]$.

\noindent\textit{(iv) Contraction property of $\mathcal{F}$.}
Using the boundedness of $\secondOperator^{-1}$ and the Lipschitz-continuity of $\operator$ we obtain for $\mathcal{F}:=\secondOperator^{-1}\circ\operator$
\[
\|\mathcal{F}[\deformation]-\mathcal{F}[\tilde\deformation]\|_{H^{2m}(\domain)}\leq C\|\operator[\deformation]-\operator[\tilde\deformation]\|_{\testSpaceDual}
\leq C(C_\image c_\image+c_\image^2)\|\deformation-\tilde\deformation\|_{H^{2m}(\domain)}
\]
for $\deformation,\tilde\deformation\in B_{\epsilon}(\Id)$,
which proves that $\mathcal{F}$ is contractive for sufficiently small $C_\image$, $c_\image$ and $\epsilon$.

Next, we prove $\mathcal{F}:B_{\epsilon}(\Id)\rightarrow B_{\epsilon}(\Id)$ for a proper choice of $C_\image$, $c_\image$ and $\epsilon$.
By using the boundedness of $\secondOperator^{-1}$ and $\secondOperator[\Id]=0$ one can infer
\[
\|\mathcal{F}[\Id]-\Id\|_{H^{2m}(\domain)}=\|\secondOperator^{-1}\circ\operator[\Id]-\secondOperator^{-1}\circ\secondOperator[\Id]\|_{H^{2m}(\domain)}
\leq C\|(\operator-\secondOperator)[\Id]\|_{H^{-2m}(\domain)}
\leq C(C_\image c_\image+c_\image^2)\,.
\]
Thus, for any $\deformation\in B_{\epsilon}(\Id)$ one gets
\begin{align*}
&\quad\|\mathcal{F}[\deformation]-\Id\|_{H^{2m}(\domain)}\leq\|\mathcal{F}[\deformation]-\mathcal{F}[\Id]\|_{H^{2m}(\domain)}+\|\mathcal{F}[\Id]-\Id\|_{H^{2m}(\domain)}\\
&\leq C(C_\image c_\image+c_\image^2)\|\deformation-\Id\|_{H^{2m}(\domain)}+C(C_\image c_\image+c_\image^2)
\leq C(C_\image c_\image+c_\image^2)\epsilon+C(C_\image c_\image +c_\image^2)\,.
\end{align*}
Now, choosing $C_\image$, $c_\image$ small enough and $\epsilon$ such that
the conditions in \eqref{eq:controlDerivatives} are satisfied
for any $\deformation\in B_{\epsilon}(\Id)$ 
and for $\deformation_1$,
$\mathcal{F}$ maps $B_\epsilon(\Id)$ onto $B_\epsilon(\Id)$.

Hence, the application of Banach's fixed point theorem proves the existence of a unique deformation 
$\deformation_2$ in $B_{\epsilon}(\Id)\subset\admset$ solving \eqref{eq:ELEII}. 
Then, the unique image $\image_2$ associated with $(\deformation_1,\deformation_2)$ can be computed using the formula \eqref{eq:pointwiseImageUpdatealternative}.
Thus, there exists a solution $(\image_2,\deformation_1,\deformation_2)\in H^1(\domain)\times\admset\times\admset$
of \eqref{eq:ELEIImage1}, \eqref{eq:ELEIImage2} and \eqref{eq:ELEI}, and this solution is unique in a small neighborhood around $(\image_0,\Id,\Id)$.
\end{proof}
\begin{theorem}[Local uniqueness and well-posedness of the discrete exponential map]
Let $2m-\frac{n}{2}>2$ and $\image_0\in H^1(\domain)$.
Then there exist neighborhoods $\imageNeighborhood\subset H^1(\domain)$ of $\image_0$ and 
$\deformationNeighborhood\subset\admset$ of $\Id$ such that for every $\image_2\in\imageNeighborhood$
there exists at most one solution $(\image_1,\deformation_1,\deformation_2)\in \imageNeighborhood\times\deformationNeighborhood\times\deformationNeighborhood$ 
of the equations \eqref{eq:ELEIImage1}-\eqref{eq:ELEI}. In particular, the discrete exponential map is locally well-posed and
\[
\image_2=\Exp{2}_{\image_0}(\image_1-\image_0)\,.
\]
\end{theorem} 
\begin{proof}
At first, we get rid of the unknown image $\image_1$. 
To this end, we observe that the sum of the two matching terms in $\energy^D[\image_0,\image_1,\deformation_1]+
\energy^D[\image_1,\image_2,\deformation_2]$ can be rewritten as follows:
\[
\int_\domain(\image_1\circ\deformation_1-\image_0)^2+(\image_2\circ\deformation_2-\image_1)^2\d x
=\int_\domain(\image_1\circ\deformation_1-\image_0)^2+(\image_2\circ\deformation_2\circ\deformation_1-\image_1\circ\deformation_1)^2\det D\deformation_1\d x\,.
\]
Therefore, the image $\image_1$ minimizing the above integral is characterized pointwise a.e. on $\domain$ by 
\[
\image_1\circ\deformation_1=\frac{\image_0+(\image_2\circ\deformation_2\circ\deformation_1)\det D\deformation_1}{1+\det D\deformation_1}
\]
and can thus be written as a function $\image_1(\image_2,\deformation_1,\deformation_2)$ of the image $\image_2$ and the  deformations $\deformation_1$ and $\deformation_2$ (we omit the dependence on the image $\image_0$).
Hence, the Euler--Lagrange equations \eqref{eq:ELEIImage1}-\eqref{eq:ELEI} can be reformulated as 
\[
0=\mathcal{K}[\image_2,\deformation_1,\deformation_2]:=\partial_{(\deformation_1,\deformation_2)}\IFTFunctional[\image_2,\deformation_1,\deformation_2]\,,
\]
where $\IFTFunctional$ is a functional on $H^1(\domain) \times \admset \times \admset$ with
\begin{align*}
\IFTFunctional[\image_2, \deformation_1, \deformation_2]&=\energy^D[\image_0,\image_1(\image_2,\deformation_1,\deformation_2),\deformation_1]+
\energy^D[\image_1(\image_2,\deformation_1,\deformation_2),\image_2,\deformation_2]\\
&=
\int_\domain
|D\deformation_1-\Id|^2+\gamma|\Delta^m\deformation_1|^2+|D\deformation_2-\Id|^2+\gamma|\Delta^m\deformation_2|^2
+\frac{\detTerm(D\deformation_1)}{\delta}(\image_2\circ\deformation_2\circ\deformation_1-\image_0)^2\d x
\end{align*}
for $\detTerm(A) = \frac{\det A}{1+\det A}$ and $\mathcal{K} : H^1(\domain) \times \admset \times \admset \to (\testSpace\times \testSpace)'$.
Now, we will show that in a neighborhood of $(\image_0, \Id, \Id)$ one obtains an explicit representation $(\deformation_1,\deformation_2)[\image_2]$ 
for the implicit equation $0 =  \mathcal{K}[\image_2, \deformation_1, \deformation_2]$ via the implicit function theorem.
Hence, for every $\image_2$, which is close to $\image_0$ in $H^1(\domain)$, there exists in a small neighborhood of $\admset \times \admset$ a unique
tuple $(\deformation_1,\deformation_2)$, which solves the above implicit equation. This indeed proves the claim.
To apply the implicit function theorem, we have to show that
$\partial_{(\deformation_1,\deformation_2)}\mathcal{K}[\image_0,\Id,\Id]=
\partial^2_{(\deformation_1,\deformation_2)}\IFTFunctional[\image_0,\Id,\Id]$
is invertible with bounded inverse.
At first, we compute the different components of $\partial^2_{(\deformation_1,\deformation_2)}\IFTFunctional[\image_2, \deformation_1, \deformation_2]$.
For this reason, we focus here on the variation of 
\[
\widetilde\IFTFunctional[\image_2,\deformation_1,\deformation_2]=
\int_\domain\detTerm(D\deformation_1)(\image_2\circ\deformation_2\circ\deformation_1-\image_0)^2\d x\,,
\]
the derivatives of the other components of $\IFTFunctional$ are straightforward.
We use integration by parts to avoid derivatives of the involved image intensities.
For the first variation with respect to $\deformation_1$ and $\deformation_2$ we obtain
\begin{align*}
\partial_{\deformation_1}\widetilde\IFTFunctional[\image_2,\deformation_1,\deformation_2](\testDeformation)
&=\int_\domain(\image_2\circ\deformation_2\circ\deformation_1-\image_0)^2 D\detTerm(D\deformation_1):D\testDeformation\\[-1.5ex]
&\hspace{2.5em}+2\detTerm(D\deformation_1)(\image_2\circ\deformation_2\circ\deformation_1-\image_0)
\nabla(\image_2\circ\deformation_2\circ\deformation_1)\cdot(D\deformation_1)^{-1}\testDeformation\d x\\
&=\int_\domain(\image_2\circ\deformation_2\circ\deformation_1-\image_0)^2 D\detTerm(D\deformation_1):D\testDeformation
-(\image_2\circ\deformation_2\circ\deformation_1)^2\div\left(\detTerm(D\deformation_1)(D\deformation_1)^{-1}\testDeformation\right)\\[-1.5ex]
&\hspace{2.5em}+2(\image_2\circ\deformation_2\circ\deformation_1)\div\left(\image_0\detTerm(D\deformation_1)(D\deformation_1)^{-1}\testDeformation\right)\d x\,,\\
\partial_{\deformation_2}\widetilde\IFTFunctional[\image_2,\deformation_1,\deformation_2](\testDeformation)
&=\int_\domain 2\detTerm(D\deformation_1)(\image_2\circ\deformation_2\circ\deformation_1-\image_0)\nabla(\image_2\circ\deformation_2\circ\deformation_1)\cdot(D(\deformation_2\circ\deformation_1))^{-1}
(\testDeformation\circ\deformation_1)\d x\\
&=\int_\domain 2(\image_2\circ\deformation_2\circ\deformation_1)
\div\left(\image_0\detTerm(D\deformation_1)(D(\deformation_2\circ\deformation_1))^{-1}(\testDeformation\circ\deformation_1)\right) \\[-1.5ex]
&\hspace{2.5em}-(\image_2\circ\deformation_2\circ\deformation_1)^2
\div\left(\detTerm(D\deformation_1)(D(\deformation_2\circ\deformation_1))^{-1}(\testDeformation\circ\deformation_1)\right)\d x
\end{align*}
using the following different versions of the chain rule:
\begin{align*}
(\nabla(\image_2\circ\deformation_2\circ\deformation_1))^T&=(\nabla(\image_2\circ\deformation_2)\circ\deformation_1)^T D\deformation_1\,,\\
\nabla(\image_2\circ\deformation_2\circ\deformation_1)^2&=2(\image_2\circ\deformation_2\circ\deformation_1)\nabla(\image_2\circ\deformation_2\circ\deformation_1)\,,\\
(\nabla(\image_2\circ\deformation_2\circ\deformation_1))^T&=(\nabla\image_2\circ(\deformation_2\circ\deformation_1))^TD(\deformation_2\circ\deformation_1)\,.
\end{align*}
Then, for the second order variations one gets
\begin{align*}
\partial_{\deformation_1}^2\widetilde\IFTFunctional[\image_2,\deformation_1,\deformation_2](\testDeformation,\testDeformationTwo)
&=\int_\domain  2(\image_2\circ\deformation_2\circ\deformation_1-\image_0) \nabla(\image_2\circ\deformation_2)\circ\deformation_1\cdot\testDeformationTwo(D\detTerm(D\deformation_1):D\testDeformation) \\[-1.5ex]
&\hspace{2.5em} +(\image_2\circ\deformation_2\circ\deformation_1-\image_0)^2 D^2\detTerm(D\deformation_1)(D\testDeformation, D\testDeformationTwo)\\[-0.25ex]
&\hspace{2.5em} -2(\image_2\circ\deformation_2\circ\deformation_1)\nabla(\image_2\circ\deformation_2)\circ\deformation_1\cdot\testDeformationTwo\div\left(\detTerm(D\deformation_1)(D\deformation_1)^{-1}\testDeformation\right) \\[-0.25ex]
&\hspace{2.5em} -(\image_2\circ\deformation_2\circ\deformation_1)^2\div\left(\partial_{\deformation_1}(\detTerm(D\deformation_1)(D\deformation_1)^{-1})(\testDeformationTwo)\testDeformation\right)\\[-0.25ex]
& \hspace{2.5em} +2 \nabla(\image_2\circ\deformation_2)\circ\deformation_1\cdot\testDeformationTwo\div\left(\image_0\detTerm(D\deformation_1)(D\deformation_1)^{-1}\testDeformation\right) \\[-0.25ex]
& \hspace{2.5em} +2(\image_2\circ\deformation_2\circ\deformation_1)\div\left(\image_0\partial_{\deformation_1}(\detTerm(D\deformation_1)(D\deformation_1)^{-1})(\testDeformationTwo)\testDeformation\right)\d x\,, \\
\partial_{\deformation_2}^2\widetilde\IFTFunctional[\image_2,\deformation_1,\deformation_2](\testDeformation,\testDeformationTwo)
&=\int_\domain  2\nabla\image_2\circ(\deformation_2\circ\deformation_1)\cdot(\testDeformationTwo\circ\deformation_1)
\div\left(\image_0\detTerm(D\deformation_1)(D(\deformation_2\circ\deformation_1))^{-1}(\testDeformation\circ\deformation_1)\right)\\[-1.5ex]
&\hspace{2.5em} +2(\image_2\circ\deformation_2\circ\deformation_1)
\div\left(\image_0\detTerm(D\deformation_1)\partial_{\deformation_2}((D(\deformation_2\circ\deformation_1))^{-1})(\testDeformationTwo)(\testDeformation\circ\deformation_1)\right)\\[-0.25ex]
&\hspace{2.5em} -2(\image_2\circ\deformation_2\circ\deformation_1)\nabla\image_2\circ(\deformation_2\circ\deformation_1)\cdot(\testDeformationTwo\circ\deformation_1)\div\left(\detTerm(D\deformation_1)(D(\deformation_2\circ\deformation_1))^{-1}(\testDeformation\circ\deformation_1)\right)\\[-0.25ex]
&\hspace{2.5em} -(\image_2\circ\deformation_2\circ\deformation_1)^2\div\left(\detTerm(D\deformation_1)\partial_{\deformation_2}((D(\deformation_2\circ\deformation_1))^{-1})(\testDeformationTwo)(\testDeformation\circ\deformation_1)\right)\d x\,,\\
\partial_{\deformation_1}\partial_{\deformation_2}\widetilde\IFTFunctional[\image_2,\deformation_1,\deformation_2](\testDeformation,\testDeformationTwo)
&=\int_\domain  2\nabla(\image_2\circ\deformation_2)\circ\deformation_1\cdot\testDeformation
\div\left(\image_0\detTerm(D\deformation_1)(D(\deformation_2\circ\deformation_1))^{-1}(\testDeformationTwo\circ\deformation_1)\right)\\[-1.5ex]
&\hspace{2.5em}+2(\image_2\circ\deformation_2\circ\deformation_1)
\div\left(\image_0\partial_{\deformation_1}(\detTerm(D\deformation_1)(D(\deformation_2\circ\deformation_1))^{-1}(\testDeformationTwo\circ\deformation_1))(\testDeformation)\right)\\[-0.25ex]
&\hspace{2.5em}-2(\image_2\circ\deformation_2\circ\deformation_1)\nabla(\image_2\circ\deformation_2)\circ\deformation_1\cdot\testDeformation
\div\left(\detTerm(D\deformation_1)(D(\deformation_2\circ\deformation_1))^{-1}(\testDeformationTwo\circ\deformation_1)\right)\\[-0.25ex]
&\hspace{2.5em}-(\image_2\circ\deformation_2\circ\deformation_1)^2
\div\left(\partial_{\deformation_1}(\detTerm(D\deformation_1)(D(\deformation_2\circ\deformation_1))^{-1}(\testDeformationTwo\circ\deformation_1))(\testDeformation)\right)\d x\,.
\end{align*}
Evaluating the second order variational derivatives at the point $(\image_0,\Id,\Id)$ yields
\begin{align*}
\partial_{\deformation_1}^2\widetilde\IFTFunctional[\image_0,\Id,\Id](\testDeformation,\testDeformationTwo)
&=\int_\domain -\image_0 \nabla\image_0\cdot\testDeformationTwo\div(\testDeformation)
-\image_0^2\div(\partial_{\deformation_1}(\detTerm(D\deformation_1)(D\deformation_1)^{-1})|_{\deformation_1 = \Id}(\testDeformationTwo)\testDeformation)\\[-1.5ex]
&\hspace{2.5em} +\nabla\image_0\cdot\testDeformationTwo\div(\image_0\testDeformation)
+2\image_0\div(\image_0\partial_{\deformation_1}(\detTerm(D\deformation_1)(D\deformation_1)^{-1})|_{\deformation_1 = \Id}(\testDeformationTwo)\testDeformation)\d x\\
&=\int_\domain\testDeformationTwo^T\nabla\image_0\nabla\image_0^T\testDeformation\d x\,,\\
\partial_{\deformation_2}^2\widetilde\IFTFunctional[\image_0,\Id,\Id](\testDeformation,\testDeformationTwo)
&=\int_\domain\nabla\image_0\cdot\testDeformationTwo\div(\image_0\testDeformation)+\image_0\div(\image_0\partial_{\deformation_2}((D(\deformation_2\circ\deformation_1))^{-1})|_{\deformation_2 = \Id}(\testDeformationTwo)(\testDeformation\circ\deformation_1))\\[-1.5ex]
&\hspace{2.5em} -\image_0\nabla\image_0\cdot\testDeformationTwo\div(\testDeformation) 
-\tfrac{1}{2}\image_0^2\div(\partial_{\deformation_2}((D(\deformation_2\circ\deformation_1))^{-1})|_{\deformation_2 = \Id}(\testDeformationTwo)(\testDeformation\circ\deformation_1))\d x\\
&=\int_\domain\testDeformationTwo^T\nabla\image_0\nabla\image_0^T\testDeformation\d x\,,\\
\partial_{\deformation_1}\partial_{\deformation_2}\widetilde\IFTFunctional[\image_0,\Id,\Id](\testDeformation,\testDeformationTwo)
&=\int_\domain \nabla\image_0\cdot\testDeformation\div(\image_0\testDeformationTwo)
+2\image_0\div(\image_0\partial_{\deformation_1}(\detTerm(D\deformation_1)(D(\deformation_2\circ\deformation_1))^{-1}(\testDeformationTwo\circ\deformation_1))|_{\deformation_1 = \Id}(\testDeformation))
\\[-1.5ex]
&\hspace{2.5em} -\image_0\nabla\image_0\cdot\testDeformation\div(\testDeformationTwo) 
-\image_0^2\div(\partial_{\deformation_1}(\detTerm(D\deformation_1)(D(\deformation_2\circ\deformation_1))^{-1}(\testDeformationTwo\circ\deformation_1))|_{\deformation_1 = \Id}(\testDeformation))\d x\\
&=\int_\domain\testDeformationTwo^T\nabla\image_0\nabla\image_0^T\testDeformation\d x\,.
\end{align*}
Here, we have used the following identities, which rely on integration by parts,
\begin{align*}
\int_\domain \image_0^2\div v-2\image_0 \div(\image_0 v)\d x
&=\int_\domain-\nabla\image_0^2\cdot v-2\image_0\div(\image_0 v)\d x\\
&=\int_\domain -2\nabla\image_0\cdot(\image_0 v)-2\image_0\div(\image_0 v)\d x=
\int_\domain 2\image_0\div(\image_0 v)-2\image_0\div(\image_0 v)\d x = 0
\end{align*} 
for any vector field $v\in H^1_0(\domain,\R^n)$.
Altogether, taking also into account the second order variation  of the remaining terms of $\IFTFunctional$ we obtain 
\begin{align*}
\partial^2_{(\deformation_1,\deformation_2)} \IFTFunctional[\image_0, \Id,\Id]((\testDeformation_1,\testDeformation_2),(\testDeformationTwo_1,\testDeformationTwo_2))
=\int_\domain & 2\gamma\Delta^m\testDeformation_1\cdot\Delta^m \testDeformationTwo_1+2D\testDeformation_1:D\testDeformationTwo_1+2\gamma\Delta^m\testDeformation_2\cdot\Delta^m\testDeformationTwo_2\\
&+2D\testDeformation_2:D\testDeformationTwo_2+\frac{1}{\delta}(\testDeformation_1+\testDeformation_2)^T\nabla\image_0\nabla\image_0^T(\testDeformationTwo_1+\testDeformationTwo_2)\d x\,.
\end{align*}
It is straightforward to verify that $\partial^2_{(\deformation_1,\deformation_2)} \IFTFunctional[\image_0,\Id,\Id]$ 
is a continuous bilinear form on $\testSpace\times \testSpace$ by taking into account the estimate 
\[
\left|\int_\domain(\testDeformation_1+\testDeformation_2)^T\nabla\image_0\nabla\image_0^T(\testDeformationTwo_1+\testDeformationTwo_2)\d x\right|
\leq C \|\image_0\|_{H^1(\domain)}^2\|(\testDeformation_1,\testDeformation_2)\|_{H^{2m}(\domain)}\|(\testDeformationTwo_1,\testDeformationTwo_2)\|_{H^{2m}(\domain)}\,,
\]
the coercivity follows by analogous arguments as in the proof of the coercivity of $\secondOperator$ (\cf Theorem \ref{thm:existenceELE}).
Thus, the Lax-Milgram Theorem ensures the required invertibility.
\end{proof}

\section{Spatial discretization and fixed point algorithm}\label{sec:algo}
In what follows, we introduce a spatial discretization scheme as well as an algorithm to compute the discrete exponential map based on the time discrete operator 
$\Exp{2}_{\image_0}(\image_1-\image_0)$ for given images $\image_0$ and $\image_1$. 
Let us recall that the computation of $\Exp{k}$ for $k>2$ requires the iterative application of $\Exp{2}$ as defined in \eqref{eq:recursiveDefExp}.
In explicit, we ask for a numerical approximation of
the matching deformations $\deformation_1$, $\deformation_2$ and the actual succeeding image
$\image_2 = \Exp{2}_{\image_0}(\image_1-\image_0)$ along the shot discrete path.
Here, we restrict to two dimensional images and for the sake of simplicity we assume that the image domain is the unit square, \ie $\domain=(0,1)^2$.
Conceptually, the generalization to three dimensions is straightforward.
As a simplification for the numerical implementation, we restrict to the case $m=1$ despite the theoretical requirement that $m>1+\frac{n}{4}=\frac{3}{2}$. 
Below, we will introduce the space of tensor product cubic splines for the discretization of deformations. For such discrete deformations the reformulation in 
Lemma~\ref{Lemma:reform}~\ref{item:secondReformulation} holds true (the regularity result in Proposition~\ref{prop:bdryRegularity}
is only required for the reformulation in the spatially continuous case).
We experimentally observed that the spatially discretized model ensures sufficient regularity of the deformations to reliably solve the Euler--Lagrange equations numerically.
To sum up, the discrete energy density that we will employ in all numerical computations is given by
\[
\energy^D[\image,\tilde\image,\deformation]=
\int_\domain|D\deformation-\Id|^2+\gamma\Delta\deformation\cdot\Delta\deformation+\frac{1}{\delta}(\tilde\image\circ\deformation-\image)^2\d x
\]
for $\image,\tilde\image\in H^1(\domain)$ and $\deformation\in\admset$.

The algorithm to compute $\deformation_2$ is based on a spatially discrete fixed point iteration similar to the one used in proof of Theorem~\ref{thm:existenceELE}.
In explicit, we follow the derivation of the fixed point mapping in this proof using now the reformulation \ref{item:secondReformulation} instead of \ref{item:firstReformulation}
in Lemma~\ref{Lemma:reform} as a starting point and define
\begin{align*}
\widetilde\operator[\deformation](\testDeformation)=
\int_\domain& 
2\gamma\Delta\deformation_1\cdot\Delta(((D\deformation)^{-1}\testDeformation)\circ\deformation_1)+
2D\deformation_1:D(((D\deformation)^{-1}\testDeformation)\circ\deformation_1)\\
&
-\frac{1}{\delta}\frac{(\image_1\circ \deformation_1-\image_0)^2}{\det D\deformation_1} \left((D\deformation)^{-T}:(D^2\deformation(D\deformation)^{-1}\testDeformation)-(D\deformation)^{-T}:D\testDeformation\right)\circ\deformation_1\d x\\
=\int_\domain&-2\gamma D\Delta\deformation_1:(D((D\deformation)^{-1}\testDeformation)\circ\deformation_1)-2\Delta\deformation_1\cdot((D\deformation)^{-1}\testDeformation)\circ\deformation_1\\
&-\frac{1}{\delta}\frac{(\image_1\circ \deformation_1-\image_0)^2}{\det D\deformation_1} \left((D\deformation)^{-T}:(D^2\deformation(D\deformation)^{-1}\testDeformation)-(D\deformation)^{-T}:D\testDeformation\right)\circ\deformation_1\d x
\end{align*}
for all $\testDeformation\in\testSpace$. Here, we used integration by parts to get the second equality.
This ansatz is numerically beneficial because it avoids the evaluation of gradients of image intensities. In fact,
we experimentally observed that the evaluation of the expression 
$\int_\domain(\image_1\circ\deformation_1-\image_0)(\nabla\image_1\cdot(D\deformation)^{-1}\testDeformation)\circ \deformation_1\d x$ 
appearing in the definition of $\operator$  in proof of Theorem \ref{thm:existenceELE} 
suffers from accuracy problems in the proximity of interfaces of $\image_1$ due to the approximate numerical quadrature.
To further improve the stability of the numerical algorithm with respect to the evaluation of the first integrand,
we additionally rewrite this expression by making use of $A:B=\tr(A^TB)$ as follows
\begin{align*}
\int_\domain D\Delta\deformation_1: D((D\deformation)^{-1}\testDeformation)\circ\deformation_1) \d x
&= \int_\domain D\Delta\deformation_1: ((D\deformation)^{-1}\circ\deformation_1) D(\testDeformation \circ\deformation_1) + 
D\Delta\deformation_1: D((D\deformation)^{-1}\circ\deformation_1) (\testDeformation \circ\deformation_1) \d x \\
&= \int_\domain((D\deformation)^{-T}\circ\deformation_1) D\Delta\deformation_1 :D(\testDeformation\circ\deformation_1)+
D\Delta\deformation_1:D((D\deformation)^{-1}\circ\deformation_1)(\testDeformation\circ\deformation_1)\d x\,.
\end{align*}
The second operator $\secondOperator$ is chosen identically to the one in the proof of Theorem \ref{thm:existenceELE}.
Then, taking into account the identity $\widetilde\operator[\deformation_2](\testDeformation)=\secondOperator[\deformation_2](\testDeformation)$
for all test functions $\testDeformation\in\testSpace$ the modified fixed point equation based on \eqref{eq:ELEIII} reads as
\[
\deformation^{j+1}=\secondOperator^{-1}\circ\widetilde\operator[\deformation^j]
\]
for $j\in\N$.

We use different discrete ansatz spaces for the deformations and the images.
As the discrete ansatz space for deformations we choose the
conforming space of cubic splines $\DeformationSpace \subset C^2(\domain)$. Here, $H=2^{-N}$ with $N\in\N$ denotes the grid size of the underlying
uniform and rectangular mesh, and the basis functions are vector-valued B-splines.
Moreover, we only impose the Dirichlet boundary condition $\Deformation=\Id$ on $\partial\domain$ instead of the stronger
boundary conditions $\Deformation-\Id\in H^2_0(\domain)$ for the discrete deformations $\Deformation\in\DeformationSpace$.
Indeed, we experimentally observed that these Dirichlet boundary conditions allow to reliably compute proper deformations.
The gray value images are approximated with finite element functions in the space~$\ImageSpace$ of piecewise bilinear and
globally continuous functions on $\domain$ with input intensities in the range $[0,1]$.
The underlying grid consists of uniform and quadratic cells with mesh size $h=2^{-M}$ with $M>N$, the index set of all grid nodes is denoted by $\indexImageNode$.
We take into account the usual Lagrange basis functions $\{\Theta^i\}_{i\in\indexImageNode}$ to represent image intensities $\Image \in \ImageSpace$.
In our numerical experiments we set $M=N+1$.

Now, we are in the position to define spatially discrete counterparts of the energy and the operators involved in the fixed point iteration.
We apply a Gaussian quadrature of order $5$ on both meshes.
The discrete energy for $\Image,\tilde\Image\in\ImageSpace$ and $\Deformation\in\DeformationSpace$ is defined as (\cf \eqref{eq:WEnerDefinition}) 
\begin{align*}
\Energy^D[\Image,\tilde\Image,\Deformation]
=&\sum_{c_H}\sum_{q_H}
\quadratureWeight^{c_H}_{q_H} \left((D\Deformation-\Id)(\quadraturePoint^{c_H}_{q_H}):(D\Deformation-\Id)(\quadraturePoint^{c_H}_{q_H})
+\gamma\Delta\Deformation(\quadraturePoint^{c_H}_{q_H})\cdot\Delta\Deformation(\quadraturePoint^{c_H}_{q_H})\right)\\
&+\frac{1}{\delta} \sum_{c_h}\sum_{q_h} \quadratureWeight^{c_h}_{q_h}\left(\tilde\Image(\Deformation(\quadraturePoint^{c_h}_{q_h}))-\Image(\quadraturePoint^{c_h}_{q_h})\right)^2\,,
\end{align*}
where we sum over all grid cells $c_H$ of the spline mesh and all local quadrature points within these cells indexed by $q_H$ with respect to the deformation energy and 
over all grid cells $c_h$ of the finer finite element mesh and all local quadrature points within these cells indexed by $q_h$. Here,
$(\quadratureWeight^{c_H}_{q_H},\quadraturePoint^{c_H}_{q_H})$ and $(\quadratureWeight^{c_h}_{q_h},\quadraturePoint^{c_h}_{q_h})$
are the pairs of quadrature weights and points on the spline mesh and the finite element mesh, respectively.
For the fully discrete counterparts of the operators $\widetilde\operator$ and $\secondOperator$ one gets
\begin{align*}
\widetilde\Operator[\Deformation](\TestDeformation)
=&\sum_{c_H}\sum_{q_H} \quadratureWeight^{c_H}_{q_H} \Big(
-2\gamma((D\Deformation)^{-T}\circ\Deformation_1(\quadraturePoint^{c_H}_{q_H})) 
D\Delta\Deformation_1(\quadraturePoint^{c_H}_{q_H}) :D(\TestDeformation\circ\Deformation_1(\quadraturePoint^{c_H}_{q_H})) \\
&\hspace*{12.5ex}-2\gamma D\Delta\Deformation_1(\quadraturePoint^{c_H}_{q_H}):
D((D\Deformation)^{-1}\circ\Deformation_1(\quadraturePoint^{c_H}_{q_H}))(\TestDeformation\circ\Deformation_1(\quadraturePoint^{c_H}_{q_H}))\\
&\hspace*{12.5ex} -2\Delta\Deformation_1(\quadraturePoint^{c_H}_{q_H})\cdot((D\Deformation)^{-1}\TestDeformation)\circ\Deformation_1(\quadraturePoint^{c_H}_{q_H}) \Big)\\
&-\sum_{c_h}\sum_{q_h}\frac{\quadratureWeight^{c_h}_{q_h}}{\delta}\frac{(\Image_1\circ \Deformation_1(\quadraturePoint^{c_h}_{q_h})-\Image_0(\quadraturePoint^{c_h}_{q_h}))^2}{\det D\Deformation_1(\quadraturePoint^{c_h}_{q_h})} \\
& \hspace*{14ex} 
\cdot \left((D\Deformation)^{-T}:(D^2\Deformation(D\Deformation)^{-1}\TestDeformation)-(D\Deformation)^{-T}:D\TestDeformation\right)\circ\Deformation_1(\quadraturePoint^{c_h}_{q_h})\,, \\[1em]
\SecondOperator[\Deformation](\TestDeformation)=&
\sum_{c_H}\sum_{q_H}\quadratureWeight^{c_H}_{q_H}\left(2\gamma\Delta\Deformation(\quadraturePoint^{c_H}_{q_H})\cdot\Delta\TestDeformation(\quadraturePoint^{c_H}_{q_H})
+2D\Deformation(\quadraturePoint^{c_H}_{q_H}):D\TestDeformation(\quadraturePoint^{c_H}_{q_H})\right)
\end{align*}
for $\TestDeformation\in\DeformationSpace$ with $\TestDeformation=0$ on $\partial\domain$.
Finally, one obtains the following fixed point iteration to compute the spatially discrete $\Deformation_2$ 
\begin{equation}
\Deformation^{j+1}=\SecondOperator^{-1}\circ\widetilde\Operator[\Deformation^j]
\label{eq:discreteFixPoint}
\end{equation}
for all $j\geq 0$ and initial data $\Deformation^{0}=\Id$. The application of $\SecondOperator^{-1}$ requires the solution of the associated linear system of equations.

In a preparatory step, the deformation $\Deformation_1\in\argmin_{\Deformation\in\DeformationSpace}\Energy^D[\Image_0,\Image_1,\Deformation]$,
which is used in the first step of a time discrete geodesic shooting,
is calculated using a Fletcher-Reeves nonlinear conjugate gradient descent multilevel scheme with an Armijo step size control.

Then, the deformation in the current step is computed using the fixed point iteration \eqref{eq:discreteFixPoint}, which
is stopped if the $L^\infty$-difference of the deformations in two consecutive iterations is below the threshold value $\mathrm{THRESHOLD}=10^{-12}$.
To compute $\Image_2$, we employ the spatially discrete analog of the update formula \eqref{eq:pointwiseImageUpdatealternative}
\begin{equation}
\Image_2(\quadraturePoint)=\left(\frac{\Image_1-\Image_0\circ\Deformation_1^{-1}}{\det(D\Deformation_1)
\circ\Deformation_1^{-1}}\right)\circ\Deformation_2^{-1}(\quadraturePoint)+\Image_1\circ\Deformation_2^{-1}(\quadraturePoint)\,.
\label{eq:updateFormulaNumeric}
\end{equation}
Here, we evaluate \eqref{eq:updateFormulaNumeric} at all grid nodes of the finite element grid.
To compute approximate inverse deformations $\Deformation_i^{-1}\in\DeformationSpace$, $i\in\{1,2\}$,
all cells of the grid associated with $\DeformationSpace$ are traversed and
the deformed positions $\Deformation_i(\mathbf{x}_j)$ for all vertices $\mathbf{x}_j$, $j\in\{1,\ldots,4\}$, of the current element are computed. 
Then, we use a bilinear interpolation of these deformed positions to define an approximation of $\Deformation_i^{-1}(x)$ for $x \in \domain$.
Furthermore, we explicitly ensure the boundary condition $\Deformation_i^{-1}(x)=x$ for $x \in \partial \domain$.

In our numerical experiments on real image data, we observed slight local oscillations emerging from the inexact evaluation of the
expression $\mathbf{J}_k=\Image_k(\quadraturePoint)-\Image_{k-1}\circ\Deformation_k^{-1}(\quadraturePoint)$ in the quadrature of the intensity modulation. 
Since the calculation of $\Exp{k}$ requires a recursive application of $\Exp{2}$, these oscillations turn out to be sensitive to error propagation,
and it is advantageous to apply in a post-processing step one iteration of the anisotropic diffusion filter
$(\Mass+\tau\Stiff[\mathbf{J}_k,\lambda])^{-1}\Mass$ to $\mathbf{J}_k$ for weight parameters $\tau,\lambda>0$ (see \cite{PeMa90}).
Here, $\Mass$ is the usual mass matrix and $\Stiff[\mathbf{J}_k,\lambda]$ the anisotropic stiffness matrix associated with $\ImageSpace$, \ie
$(\Stiff[\mathbf{J}_k,\lambda])_{i,j}=\sum_{c_h,q_h}
\quadratureWeight^{c_h}_{q_h}(1+\lambda^{-2}\|\nabla\mathbf{J}_k(\quadraturePoint^{c_h}_{q_h})\|^{2})^{-1} \nabla\Theta^i(\quadraturePoint^{c_h}_{q_h})\cdot\nabla\Theta^j(\quadraturePoint^{c_h}_{q_h})$
for $i,j\in\indexImageNode$.
Furthermore, in all following applications except the first test case (Figure~\ref{figure:circleExample}) we choose
$\tau=\beta^{k-2}\cdot10^{-3}$ as the exponentially decaying time step size ($k$ denoting the index of the image in the sequence and $\beta=0.8$)
and $\lambda=0.5$ as the smoothing parameter along the discrete geodesic. The impact of this filtering can be seen for instance in Figure~\ref{figure:women}.

\section{Numerical results}\label{sec:results}
In this section, we present applications of the fully discrete exponential map.
In all computations, we use the parameters $\gamma=10^{-4}$ and $\delta=10^{-2}$. 

\setlength{\unitlength}{0.05\linewidth}
\begin{figure}[htb]
\centering
\includegraphics[width=\linewidth]{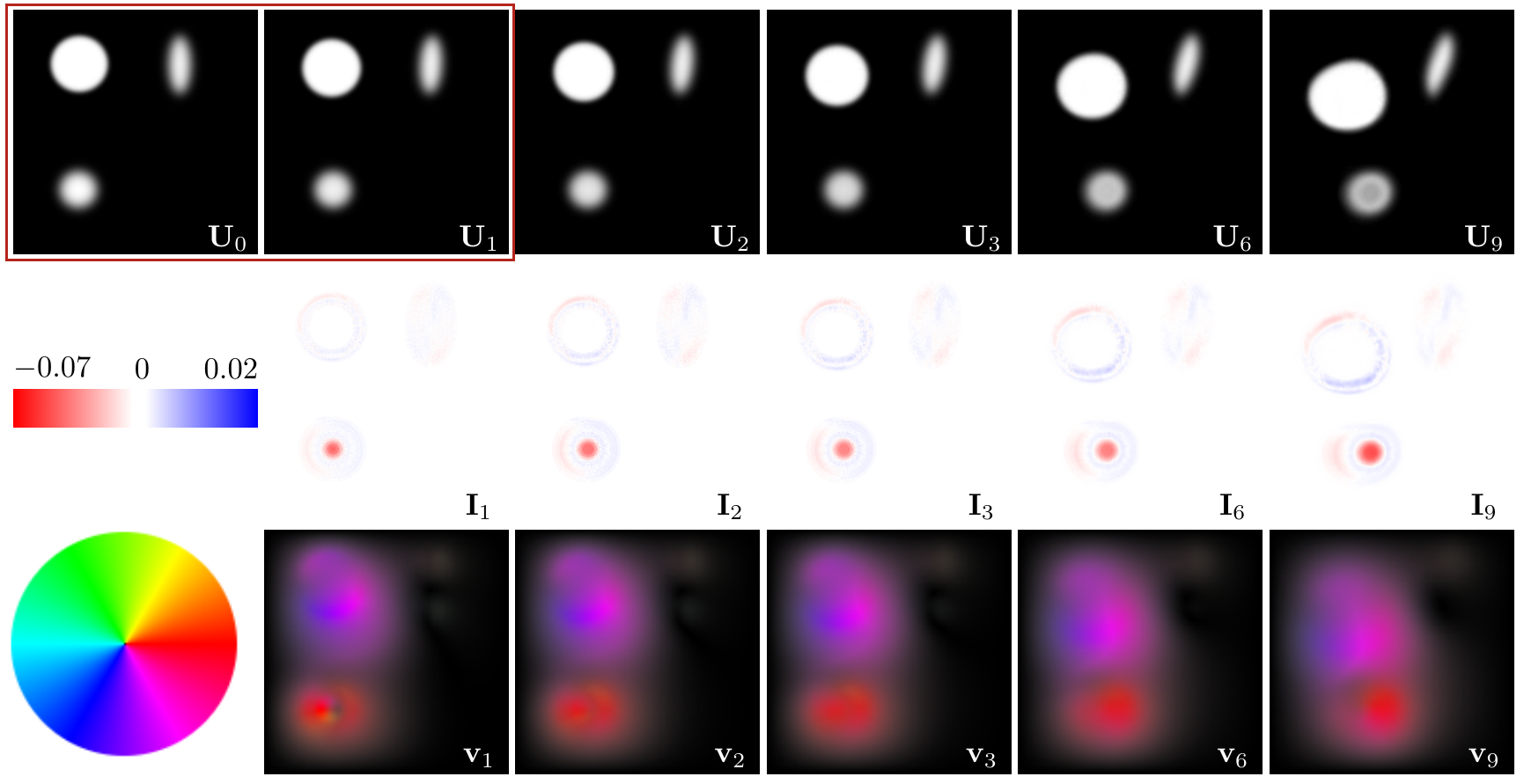}
\caption{First row: The discrete exponential map $\Exp{k}_{\Image_0}(\Image_1-\Image_0)$ with $k=0,1,2,3,6,9$ for images showing three ellipses
(input images are framed in red).
Second row: the associated intensity modulations $\IntensityModulation_k$. Third row:
the discrete velocity fields~$\Velocity_k$ (the hue refers to the direction, the intensity is proportional to its norm).}
\label{figure:circleExample}
\end{figure}
As a first example, we investigate an artificial test case consisting of an input image $\Image_0$ with three ellipses of different intensities and an associated variation $\Image_1-\Image_0$.
The first row in Figure \ref{figure:circleExample} depicts distinct images of the image sequence $\Exp{k}_{\Image_0}(\Image_1-\Image_0)$
for time steps $k=0,1,2,3,6,9$, the input images $\Image_0$ and $\Image_1$ with resolution $257\times 257$
are framed in red. 
In the initial variation $\Image_1-\Image_0$ underlying the exponential shooting, the upper left ellipse is slightly translated to the bottom and simultaneously expanded.
The upper right ellipse is undergoing a small rotation and the third one is also slightly translated with some modulation of the shading.
The initial variation encoded in the image pair $(\Image_0, \Image_1)$ is prolongated along the sequence generated by an iterative application of the discrete exponential map $\Exp{2}$.
In the second row, for each $k>0$ the discrete intensity modulations $\IntensityModulation_k=\Image_k\circ\Deformation_k-\Image_{k-1}$ are visualized.
Here, on the left the color bar with bounds coinciding with the extremal values for this image sequence is displayed.
The third row depicts the discrete velocity fields $\Velocity_k=\frac{1}{\tau}(\Deformation_k-\Id)$ for each $k$,
where $\tau=\frac{1}{K}$ is the associated time step size. Here, the hue refers to the direction 
and the color intensity is proportional to the local norm of $\Velocity_k$ as represented by the leftmost color wheel.
In particular, one observes that the resulting underlying velocity field $\Velocity_k$ is not constant in time.

\setlength{\unitlength}{0.05\linewidth}
\begin{figure}[htb]
\includegraphics[width=\linewidth]{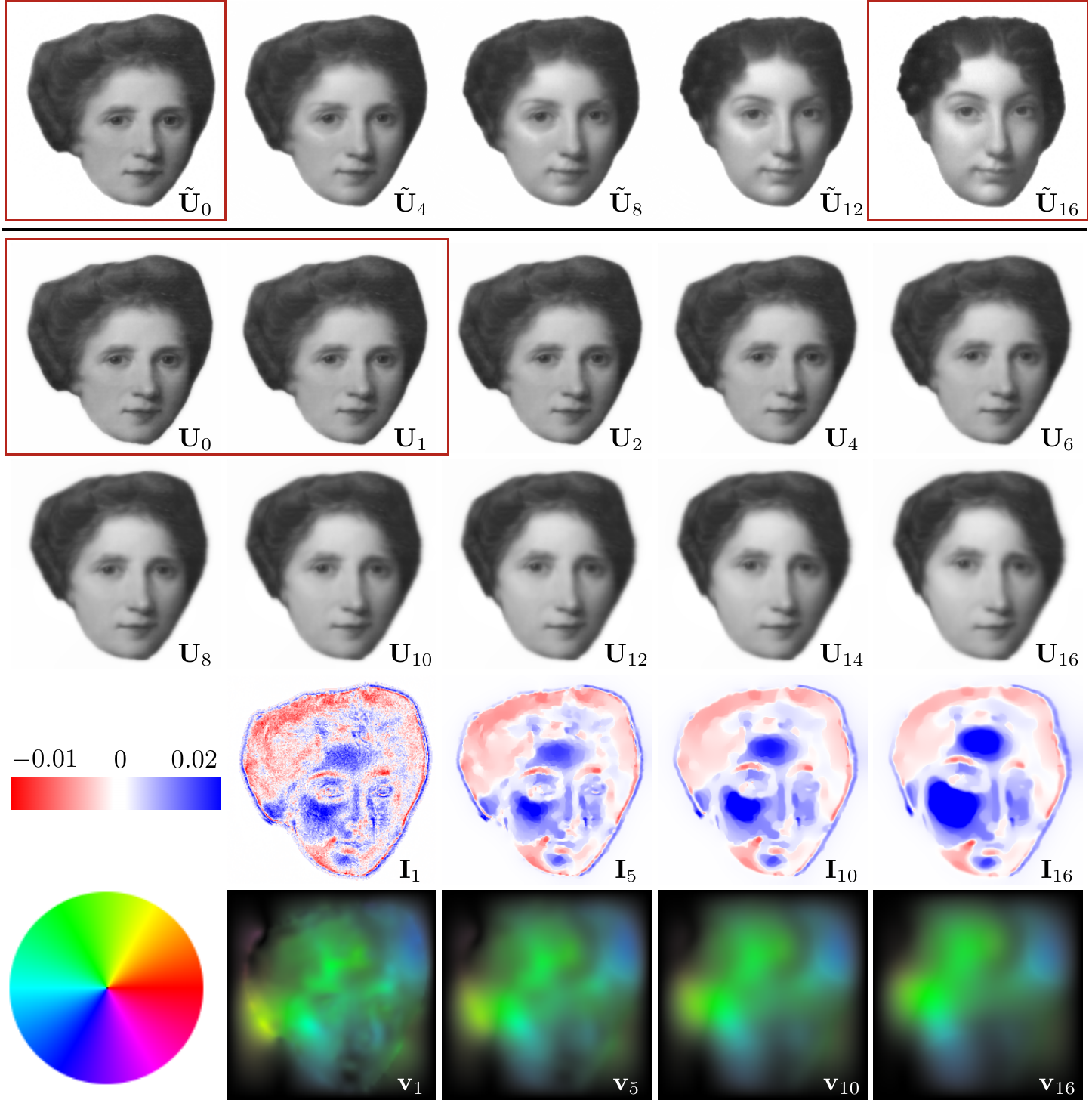}
\caption{The first row depicts distinct images of the discrete geodesic sequence associated with the  input images $\tilde\Image_0$ and $\tilde\Image_{16}$ (in red boxes).
The discrete exponential map for distinct time steps~$k$ is shown in the second and third row, where the input images $\Image_0$ and $\Image_1$
coincide with $\tilde\Image_0$ and $\tilde\Image_1$ from the geodesic sequence, respectively. In addition, the corresponding intensity modulations as well as the discrete velocity fields (fourth and fifth row)
are shown for some time steps~$k$.}
\label{figure:women}
\end{figure}
In \cite[Figure 6.2]{BeEf14}, a geodesic sequence between two female portrait
paintings\footnote{first painting by A. Kauffmann (public domain, see \nolinkurl{http://commons.wikimedia.org/wiki/File:Angelika_Kauffmann_-_Self_Portrait_-_1784.jpg}), second painting by R. Peale (GFDL, see \nolinkurl{http://en.wikipedia.org/wiki/File:Mary_Denison.jpg})}
was computed using the  finite element discretization
for both the images and the deformations on the same grid. The image resolution is $257\times 257$ ($M=8$).
We recomputed this geodesic sequence $(\tilde\Image_0,\tilde\Image_1,\ldots,\tilde\Image_{16})$ with $K=16$ for the discrete function
spaces $\ImageSpace$ and $\DeformationSpace$ with  $N=7$, the resulting sequence
is shown in the first row of Figure \ref{figure:women} with framed input images $\tilde\Image_0$ and $\tilde\Image_{16}$.
This is compared with the discrete exponential shooting for the initial image pair $\Image_0,\, \Image_1$ taken from this geodesic sequence.
One observes that the discrete exponential map is capable to recover the original geodesic sequence for small time steps $k$. Only in late stages visible differences become apparent.
Again, we highlight that the discrete motion fields significantly alter in time.

\setlength{\unitlength}{0.05\linewidth}
\begin{figure}[htb]
\includegraphics[width=\linewidth]{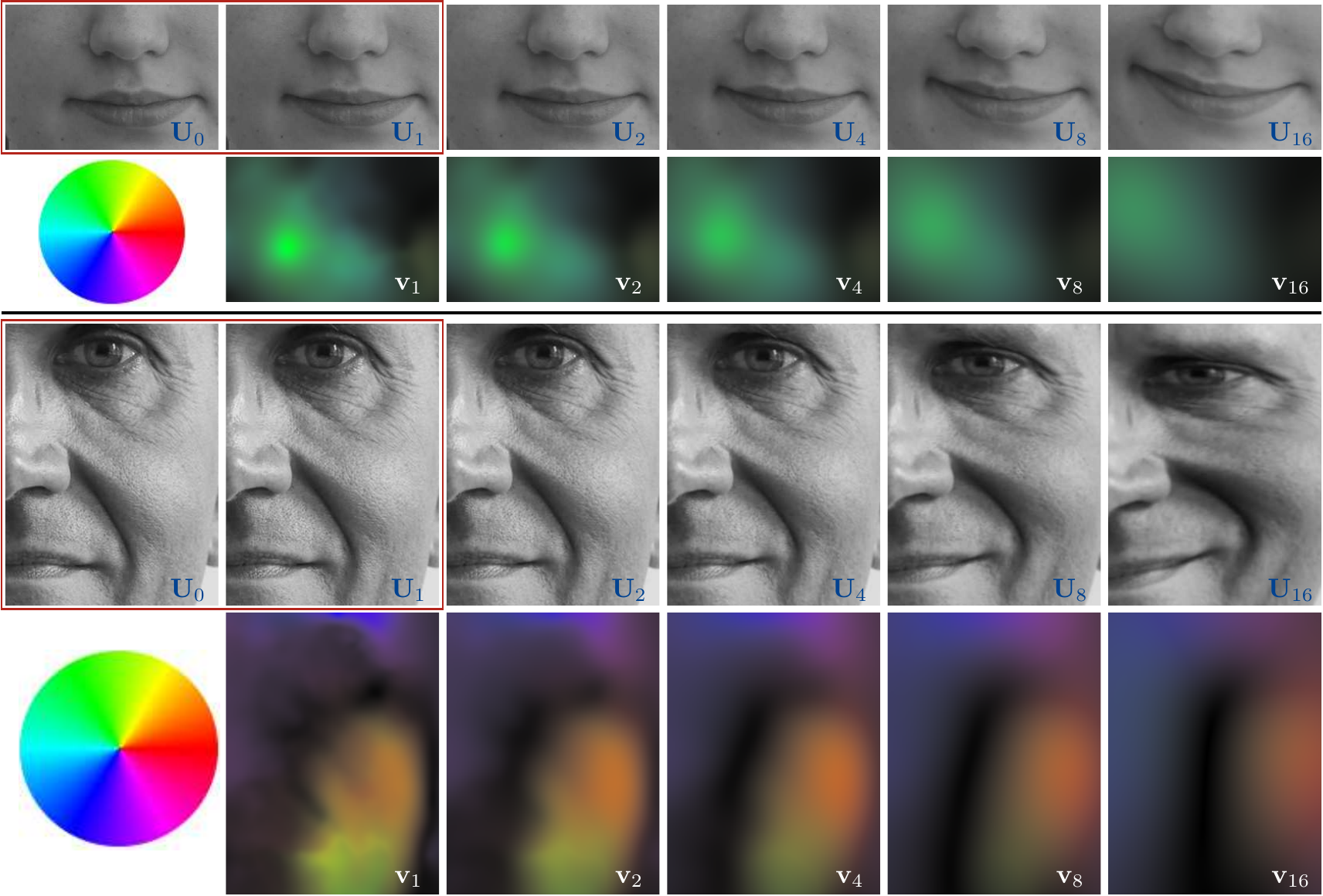}
\caption{First/third row: picture details of $\Exp{k}_{\Image_0}(\Image_1-\Image_0)$ applied to two pairs of photos of human faces for time steps $k=0,1,2,4,8,16$.
Second/fourth row: the associated discrete velocity fields $\Velocity_k$.}
\label{figure:humanFace}
\end{figure}
Figure \ref{figure:humanFace} depicts a picture details for the time steps $k=0,1,2,4,8,16$ of the discrete exponential map  applied to two different pairs of photos.
These photos show human faces and small variations of them and
the resolution of the underlying full images is $1025\times 1025$. The red boxes indicate these input images (first and third row),
which are consecutive photos of a series at $5$ and $7$ fps, respectively, taken with a digital camera.
We observe that small initial variations result in a nonlinear deformation of the lips (first row)
and of the lips, the cheeks and the eyes (third row), respectively. Furthermore, the textures are transported along the sequence.
The second and fourth row depict the color coded time varying velocity fields.

\setlength{\unitlength}{0.05\linewidth}
\begin{figure}[htb]
\includegraphics[width=\linewidth]{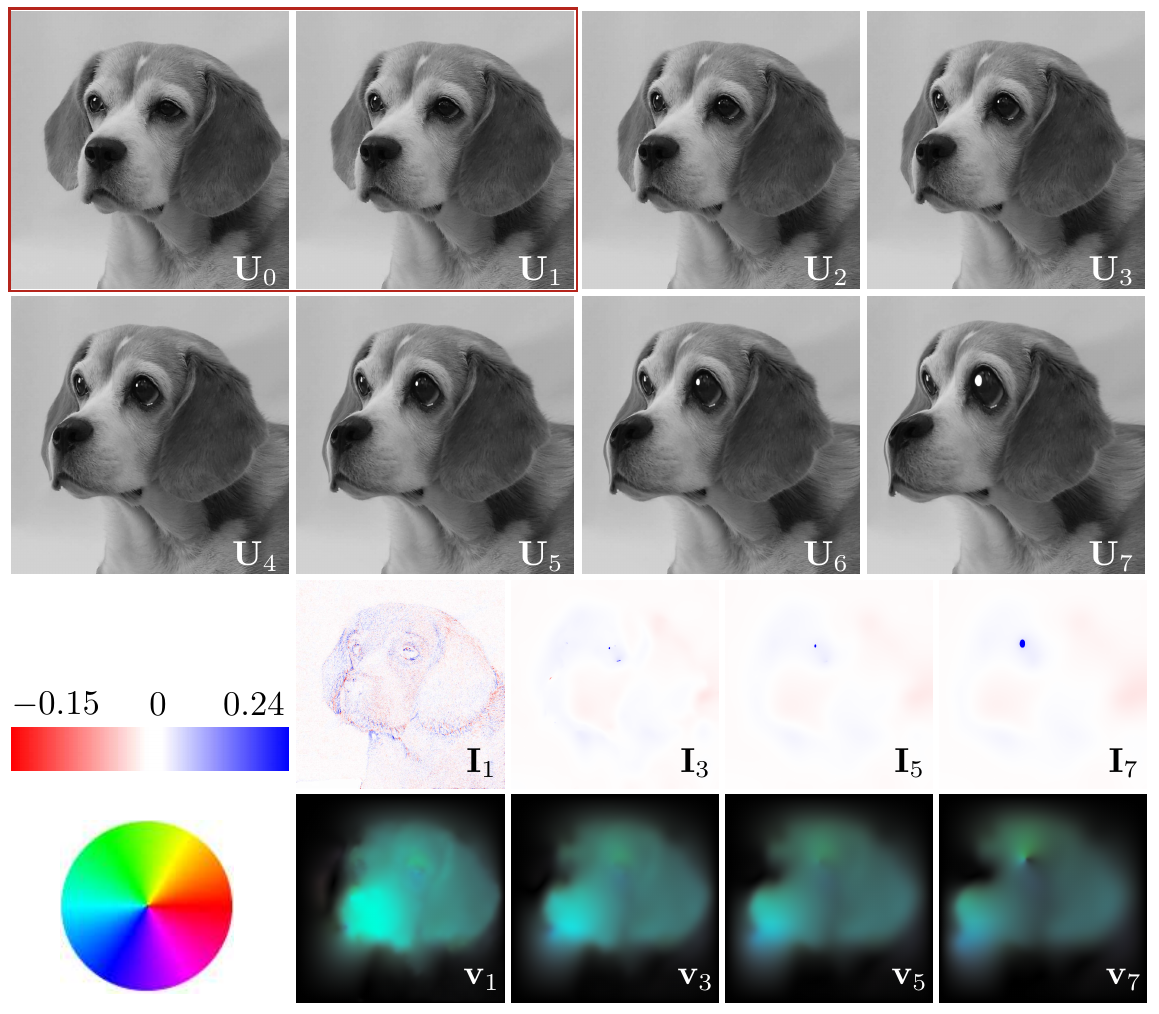}
\caption{The discrete exponential map for time steps $k=0,\ldots,7$ with two photos of a dog as initial data (first and second row).
Third/fourth row: the associated intensity modulations and velocity fields for distinct time steps.}
\label{figure:dog}
\end{figure}
Figure \ref{figure:dog} shows the discrete exponential map for $k=0,\ldots,7$ applied to a pair of images of a dog for a resolution of $1025\times 1025$.
Again, the input pictures are consecutive photos of a series with $7$ fps taken with a digital camera.
The initial image pair shows a slight rotation of the dog's head and a small opening of its eyes.
The proposed algorithm generates a extrapolation of this movement.
As a consequence in particular of the rotation, the method fills in reasonable image features below the mouth and right to the ear
which correspond to hidden object regions in $\Image_0$ and $\Image_1$.

\paragraph{Supplementary material.}
The supplementary material of this publication includes video sequences with
animations of the discrete exponential map shown in Figure~\ref{figure:women} (together with the discrete geodesic interpolation),
Figure~\ref{figure:humanFace} and Figure~\ref{figure:dog}.

\paragraph{Acknowledgements.} A. Effland and M. Rumpf acknowledge support of the Hausdorff Center for
Mathematics and the Collaborative Research Center 1060 funded by the German Research Foundation.

{\small
\bibliographystyle{abbrv}
\bibliography{EfRuSc16_MetamorphosisExp.bib}
}

\end{document}